\documentclass[11pt]{article}
\usepackage{graphicx}
\usepackage{amssymb}
\usepackage{amsmath}
\usepackage{amsthm}
\usepackage{bm}
\usepackage{tikz}
\usepackage{url}

\newtheorem{theorem}{Theorem}
\newtheorem{corollary}[theorem]{Corollary}
\newtheorem{example}[theorem]{Example}
\newtheorem{lemma}[theorem]{Lemma}
\newtheorem{definition}{Definition}

\expandafter\def\expandafter\UrlBreaks\expandafter{\UrlBreaks
  \do\a\do\b\do\c\do\d\do\e\do\f\do\g\do\h\do\i\do\j%
  \do\k\do\l\do\m\do\n\do\o\do\p\do\q\do\r\do\s\do\t%
  \do\u\do\v\do\w\do\x\do\y\do\z\do\A\do\B\do\C\do\D%
  \do\E\do\F\do\G\do\H\do\I\do\J\do\K\do\L\do\M\do\N%
  \do\O\do\P\do\Q\do\R\do\S\do\T\do\U\do\V\do\W\do\X%
  \do\Y\do\Z}

\title{Goal oriented time adaptivity using local error estimates}
\author{Peter Meisrimel$^{\mbox{\tiny\rm 1}}$, Philipp Birken$^{\mbox{\tiny\rm 1}}$}
\begin{document}
\maketitle
\baselineskip=0.9
\normalbaselineskip
\vspace{-3pt}
\begin{center}{\footnotesize\em $^{\mbox{\tiny\rm 1}}$Centre for the
    mathematical sciences, Numerical Analysis, Lund University, Lund, Sweden\\ email: peter.meisrimel\symbol{'100}na.lu.se}
\end{center}

\begin{abstract}
We consider initial value problems where we are interested in a quantity of interest (QoI) that is the integral in time of a functional of the solution of the IVP. For these, we look into local error based time adaptivity. We derive a goal oriented error estimate and timestep controller, based on error contribution to the error in the QoI, for which we prove convergence of the error in the QoI for tolerance to zero under weak assumptions. We analyze global error propagation of this method and derive guidelines to predict performance of the method. In numerical tests we verify convergence results and guidelines on method performance. Additionally, we compare with the dual-weighted residual method (DWR) and classical local error based time-adaptivity. The local error based methods show better performance than DWR and the goal oriented method shows good results in most examples, with significant speedups in some cases.
\end{abstract}
\medskip\\
{The authors gratefully acknowledge support from the Craaford Foundation under grant number 20150681.}

\section{Introduction}\label{intro}
%
%
A typical situation in numerical simulations based on differential equations is that one is not interested in the solution of the differential equation per se, but a \textit{Quantity of Interest} (QoI) that is given as a functional of the solution. For example, when designing an airplane, the QoI would be the lift coefficient divided by the drag coefficient. In simulations of the Greenland ice sheet, one would like to know the net amount of ice loss over a year. When simulating wind turbines, the amount of energy produced during a certain time period is more important than the actual flow solution. 

Further examples are found in optimization problems with ODEs or PDEs as constraints. In the turbine example, one may want to optimize blade shape or determine optimal placement of e.g. tidal turbines \cite{Funke2014} for maximal energy output. Inverse problems in e.g. oceanography \cite{CarlisleThacker1992} can also be considered. Here the aim is to determine model parameters or initial conditions to fit measurement data to goal functions of simulation results. An example for such an inverse problem is to determine vertical mixing parameters with the QoI being the total inflow of salt water from the North Sea into the Baltic Sea.
 
In this article, we restrict ourselves to problems where the QoI is given as an integral over time of a functional of the solution. From the examples above, only the steady state problem in air plane design does not qualify. The basic problem we consider is thus: Given the initial value problem
\begin{equation}\label{EQ BASE IVP}
\dot{\bm{u}}(t) = \bm{f}(t,\bm{u}(t)),
\quad t\in[t_0,t_e],
\quad \bm{u}(t_0) = \bm{u}_0,
\end{equation}
for a sufficiently smooth function $\bm{f}:[t_0, t_e] \times \mathbb{R}^d \rightarrow \mathbb{R}^d$ with solution $\bm{u}(t)$, we are interested in the QoI
\begin{equation}\label{EQ J DEF}
J(\bm{u}) := \int_{t_0}^{t_e}j(t, \bm{u}(t))dt,
\end{equation}
with $j:[t_0, t_e] \times \mathbb{R}^d \rightarrow \mathbb{R}$, which we will refer to as \textit{density function}, following the notation in \cite{bangerth2013adaptive}.

When solving PDEs, the system of ODEs originates from a semi-discretization, thus $\bm{u}$ consists of unknowns of the space discretization. Consequently $j$ can be used to provide spatial weighting and to select only specific points or regions of the spatial discretization. 

The goal here is to determine an adaptive discrete approximation $\bm{u}_h \approx \bm{u}(t)$. Our degrees of freedom are the timesteps and we want to use as few as possible. This strategy will not yield an optimal solution, but works well in practice. We adapt the timesteps $\Delta t$ using a timestep controller, which is based on local error estimates. The solution process as a whole involves a variety of schemes.

An \textit{adaptive method} consists of a \textit{time-integration} scheme for \eqref{EQ BASE IVP}, an \textit{error estimator}, a \textit{timestep controller} and an \textit{initial timestep} $\Delta t_0$. If we consider problem \eqref{EQ BASE IVP} - \eqref{EQ J DEF}, the adaptive method also includes a discrete approximation $J_h \approx J$ given by a \textit{quadrature scheme}.

The input for an adaptive method is a tolerance $\tau$, which is used in the timestep controller and possibly to determine $\Delta t_0$. The output is an approximation to the solution, this can be a discrete solution $\bm{u}_h$, or $J_h(\bm{u}_h)$, depending on which problem is considered. Since we have an adaptive method, we cannot use the usual notion of convergence for $\Delta t \rightarrow 0$ for a time-integration scheme. Instead, we consider the limit of the tolerance going to zero.
\begin{definition}
An adaptive method for an IVP \eqref{EQ BASE IVP} is called \textit{convergent} (in the solution), if
\begin{equation*}
\| \bm{u}(t_e) - \bm{u}_N\| = 0, \quad \text{for} \quad \tau \rightarrow 0,
\end{equation*}
where $\bm{u}_N$ is an approximation to $\bm{u}(t_e)$ and $\| \cdot\|$ is an appropriate norm. 

An adaptive method for an IVP \eqref{EQ BASE IVP} with QoI \eqref{EQ J DEF} is \textit{convergent in the QoI}, if 
\begin{equation*}
|J(\bm{u}) - J_h(\bm{u}_h)| = 0, \quad \text{for} \quad \tau \rightarrow 0.
\end{equation*}
\end{definition}
For a convergent adaptive method we are naturally interested in the convergence rate and will express it in terms of $\mathcal{O}(\tau^q)$. This definition of adaptive methods and convergence is targeted to local error based methods, but can also be considered for methods based on global error estimates.

For goal oriented adaptivity, the standard approach is the dual weighted residual (DWR) method \cite{becker2001_DWR,prudho:15}. Originally, it was developed for spatial problems, but has been extended to time dependent problems. The basic idea is to use the adjoint (dual) problem to get an estimate of the error in the QoI. In the time dependent case, the adjoint problem is a terminal value problem (IVP backwards in time). For linear problems, this gives rise to global error bounds, in the nonlinear case, global error estimates are obtained. 

The DWR method is based on global a-posteriori error estimates. To obtain these error estimates one needs to subsequently integrate forward and backward in time. Here the primal and adjoint solution need to be stored. The error estimate is obtained from the primal and dual solution and is used to refine the meshes. This iterative process is repeated until a discretization is found, where the error estimate $\eta(u_h)$ fulfills 
\begin{equation*}
|J(\bm{u}) - J_h(\bm{u}_h)| \approx \eta(u_h) \leq \tau.
\end{equation*}
The major drawback of this method is its cost, both in implementation and computation. To reduce computational effort, Carey et. al. suggested to apply the approach in a blockwise manner, thus making it more local \cite{CaEJLT:10}. The storage of the primal and dual solution can be problematic for high resolutions. This, for example, can be solved by check-pointing \cite{meiric:14,meiric:15}, but will further increase computational costs. The method requires a full variational formulation, restricting it to Galerkin type schemes in space and time.
 
An alternative is to use a classical time adaptive method for IVPs based on estimating the local error. Results on convergence are well established and described in standard textbooks \cite{shampine1994_NUMODE,haiwan:93}. This adaptive method is not goal oriented, but can be used to solve problems with QoIs. We do not have global error bounds, since the accumulation of local errors is hard to analyse. This approach works particularly well for stiff problems, since there, local errors typically dissipate with time.
 
We choose a different approach, aiming to get the best of both methods. To this end we derive a new error estimator for the classic adaptive method to make it goal oriented. We estimate the time-stepwise error contribution to the error in the QoI, which consists of both quadrature and time-integration errors. Neglecting the quadrature contribution, we derive a local error estimate and use it in the deadbeat controller.

We show that convergence in the QoI follows from convergence in the solution, with additional requirements on the timesteps. The derived goal oriented adaptive method fulfills these requirements and is convergent in the QoI under weak assumptions. To obtain high convergence rates in the QoI when using higher order ($> 2$) time-integration schemes, one needs solutions of sufficiently high order in all quadrature evaluation points. We explain how to obtain these from the stage value of a given RK scheme. 

We do our analysis for one-step methods for time-integration, embedded Runge-Kutta schemes \cite{haiwan:93} for error estimation, the deadbeat controller \eqref{EQ CONT DEADBEAT} and simple choices for $\Delta t_0$. These restrictions are done for easier analysis, but it is straightforward to extend the results to other error estimation techniques, such as Richardson-extrapolation \cite{haiwan:93}. For different controllers, such as PID controllers \cite{Soderlind2003}, our results allow for simple convergence proofs based on similarity to the deadbeat controller. The results hold for a wide range of initial timesteps and thus for any reasonable scheme used to compute $\Delta t_0$.

Implementation of this method only requires a standard deadbeat controller, an embedded Runge-Kutta scheme and the density function $j(t, \bm{u}(t))$. Due to being based on local error estimates, the method is computationally very cheap. For problems where the density function only regards a small part of the state vector $\bm{u}$, the error estimate will be even cheaper than the classical one. 

A similar method has been proposed by \cite{John2010,Turek1999,Wick2017}, using various other techniques for error estimation. John, Rang propose it for drag and lift coefficients in incompressible flows, but do not show numerical results \cite{John2010}. Turek describes a case where using the method for an alternating lift coefficient leads to "catastrophical results" \cite{Turek1999}. Wick uses a point-wise evaluation of the displacement field in fluid-structure interaction \cite{Wick2017,Failer}. The author describes inconsistent convergence patterns but concludes satisfying results. 

To be able to make statements on the performance of the goal oriented adaptive method, we analyse the impact of global error dynamics on the error in the QoI. This analysis revolves around the nullspace of the density function $j(t, \bm{u})$ and thus our error estimator. A method performs well, if all relevant processes are sufficiently resolved in time. To be able to sufficiently resolve a process, its local error or the local error of a faster process, must appear in the error estimate. The question if a process is relevant for the QoI is a matter of global error dynamics. Thus, with sufficient knowledge on the global error dynamics, we are able to make predictions on the performance of the goal oriented adaptive method.

We use numerical tests with widely different global error behaviors with respect to the QoI. For these we confirm the convergence results and are able to explain the performance results. It turns out to be relatively easy to predict bad performance, but hard to predict good performance. Our results show that the local error based methods are more efficient than the DWR method. The goal oriented adaptive method shows good performance in most cases and significant speedups in some.

The structure of the article is as follows: We first review current adaptive methods in section \ref{SEC OLD METHODS}, then we explain and analyse our approach in section \ref{SEC LOCAL GOAL}. Numerical results are presented in section \ref{SEC NUM RESULTS}.
%
\section{Current adaptive methods}\label{SEC OLD METHODS}
\subsection{A posteriori error estimation via the dual weighted residual method}\label{SEC DWR}
The starting point of the DWR method is an initial value problem in variational formulation: Find $u\in U$, such that
\begin{equation*}
A(u;v) = F(v), \quad u(t_0) = u_0, \quad \forall v \in V.
\end{equation*}
Here, $U$ and $V$ are appropriate spaces, $A$ is linear in $v$ and possibly nonlinear in $u$. Here we have $A(u;v) = (u_t,v) -(f(t, u),v)$ and $F(v) = 0$, see \eqref{EQ BASE IVP}. Furthermore, there is a discrete approximation to this problem, also in weak form: Find $u_h\in U_h$, such that
\begin{equation}\label{EQ DWR FWD}
A(u_h;v_h) = F(v_h), \quad u_h(t_0) = u_0, \quad \forall v_h \in V_h.
\end{equation}
Here, $U_h \subset U$ and $V_h\subset V$ are finite element spaces in time.
%
\subsubsection{The error estimate}
To obtain an estimate of the error $e^J = J(u)-J(u_h)$ in the QoI \eqref{EQ J DEF}, one uses the linearised adjoint problem for $J(u)$: Find $z \in V$, such that
\begin{equation*}
A'(u;v,z) = J'(u;v), \quad z(t_e) = 0, \quad \forall v \in U
\end{equation*}
and its discrete version
\begin{equation}\label{EQ ADJOINT DISCR}
A'(u_h;v_h,z_h) = J'(u_h;v_h), \quad z_h(t_e) = 0, \quad \forall v_h \in U_h,
\end{equation}
where $A'$ and $J'$ are the Gateaux derivatives of $A$ and $J$ with respect to $u$ in direction $v$. Note that the adjoint problem is an initial value problem backwards in time.
 
An approximation of the error in the QoI is given by
\begin{equation*}
e^J \lesssim A(u_h,z - z_h) - F(z - z_h),
\end{equation*}
with equality for linear functionals and approximate upper bounds for the general nonlinear case. Using an approximation $z_h^+ \approx z$, which is of higher accuracy than $z_h$, using e.g. higher order interpolation or a discrete solution on a finer grid \cite{bangerth2013adaptive}, one gets an estimate 
\begin{equation}
e^J \lesssim \eta(u_h) := A(u_h,z_h^+ - z_h) - F(z_h^+ - z_h).
\end{equation}
This can be further bounded by decomposing it into timestep wise contributions and thus giving a guide on where and how to adapt. For this to work, it is imperative that the solutions of the primal and adjoint problems are obtained at all points. This can cause storage problems for long time simulations and can be dealt with using check-pointing \cite{Griewank2000}. 
\subsubsection{Adaptation scheme}\label{SEC DWR ALG}
A large number of different adaptation strategies exist. Here we use a fixed-rate strategy \cite{bangerth2013adaptive}, where the $r\in[0, 1]$ elements with largest error are refined. Summarizing, the following scheme is obtained.
\begin{enumerate}
\item Start with initial grid.
\item Solve forward problem \eqref{EQ DWR FWD} to obtain $u_h$.
\item Construct and solve adjoint problem \eqref{EQ ADJOINT DISCR} to obtain $z_h$.
\item Calculate $z_h^+ \approx z$.
\item Calculate error estimate $\eta(u_h)$.
\item Check $\eta(u_h) \leq \tau$, if not met, refine grids and restart.
\end{enumerate}
The scheme is very expensive due to the need of solving adjoint problems to obtain an error estimate. While one can use generic schemes for grid adaptation, the adjoint problem and the error estimate are specific to a given equation and goal functional. Construction and solution of the adjoint problem can be automated using software such as \texttt{dolfin-adjoint} \cite{farrell2013_DOLFINADJ}. An advantage of the method is that the error estimate is global and one can expect the resulting discretizations to be of high quality.

We use a finer grid to approximate $z$ by $z_h^+$, making this the most expensive step in the computation of $\eta(u_h)$. 
%
\subsection{Time Adaptivity based on local error estimates}\label{SEC LOCAL ERR BASIC}
%
The second adaptive method we discuss is the standard in ODE solvers. It uses local error estimates of the solution and does not take into account QoIs. The results from this section for One-step methods and the deadbeat controller \eqref{EQ CONT DEADBEAT} are in principal classic \cite{shampine1985_STEP}. 

Here, we present a new convergence proof that separates requirements on the error estimate, timesteps and $\Delta t_0$, for generic One-step methods. This makes it easier to show convergence for general controllers and estimates, and we use it to show convergence in the QoI for the goal oriented adaptive method in section \ref{SEC LOCAL GOAL}.
We first introduce the relevant terminology used in this paper. For readers familiar with time adaptivity for ODEs, we use \textit{local extrapolation} and \textit{Error Per Step (EPS)} based control, see \cite{shampine1985_STEP}.
\begin{definition}
The \textit{flow} \cite{Soderlind2006a} of an IVP \eqref{EQ BASE IVP} is the map
\begin{equation*}
\mathcal{M}^{t, \Delta t} : \bm{u}(t) \rightarrow \bm{u}(t + \Delta t),
\end{equation*}
where $t \in [t_0, t_e]$ and $t + \Delta t \leq t_e$ for $\Delta t > 0$.
\end{definition}
The flow acts as the solution operator for $\bm{u}(t)$. To numerically solve an IVP means to approximate the flow by a \textit{numerical flow map} $\mathcal{N}^{t, \Delta t}: \mathbb{R}^d \rightarrow \mathbb{R}^d$ defined by some numerical scheme. A timestep can be written in the form
\begin{equation*}
\bm{u}_{n+1} = \mathcal{N}^{t_n, \Delta t_n} \bm{u}_n.
\end{equation*}
We generally assume problem \eqref{EQ BASE IVP} to have the unique solution $\bm{u}(t)$ guaranteeing existence of the flow map $\mathcal{M}^{t, \Delta t}$. 

We define the global error by
\begin{equation}\label{EQ GLOBAL ERR}
\bm{e}_{n+1} 
:= \bm{u}_{n+1} - \bm{u}(t_{n+1}) 
= \mathcal{N}^{t_n, \Delta t_n} \bm{u}_n - \mathcal{M}^{t_n, \Delta t_n} \bm{u}(t_n).
\end{equation}
By adding zero we obtain the global error propagation form
\begin{equation}\label{EQ ERR PROP NUM}
\bm{e}_{n+1} 
= \underbrace{(\mathcal{N}^{t_n, \Delta t_n} - \mathcal{M}^{t_n, \Delta t_n})\bm{u}_n}_{\text{global error increment}} + 
\underbrace{\mathcal{M}^{t_n, \Delta t_n}\bm{u}_n - \mathcal{M}^{t_n, \Delta t_n} \bm{u}(t_n)}_{\text{global error propagation}}.
\end{equation}
The dynamics of global error propagation are usually not known. The global error increments, however, have a known structure. 
\begin{definition}
Assume a sufficiently smooth right-hand side $\bm{f}$ for \eqref{EQ BASE IVP}. The \textit{principal error function} \cite{haiwan:93} $\bm{\phi}$ of a scheme $\mathcal{N}^{t, \Delta t}$ of order $p$ is
\begin{equation}\label{EQ PRINCP ERROR FUNC}
\bm{\phi}(t, \bm{u}) := \lim_{\Delta t \rightarrow 0} \frac{(\mathcal{N}^{t, \Delta t} - \mathcal{M}^{t, \Delta t})\bm{u}}{\Delta t^{p + 1}}.
\end{equation}
The \textit{local error} of a scheme $\mathcal{N}^{t, \Delta t}$ of order $p$ is
\begin{equation*}
\left(\mathcal{N}^{t_n, \Delta t_n} - \mathcal{M}^{t_n, \Delta t_n}\right)\bm{u}_n = \Delta t_n^{p+1} \bm{\phi}(t_n, \bm{u}_n) + \mathcal{O}(\Delta t_n^{p+2}).
\end{equation*}
\end{definition}

Here, the local error is equivalent to the global error increment \eqref{EQ ERR PROP NUM}. We will estimate the local error and derive a timestep controller to keep the norm of the local error in check. Then we show that the resulting adaptive method is convergent, that is, the global error can be controlled by the global error increments and goes to zero for $\tau \rightarrow 0$.
%
%
\subsubsection{Error estimation and timestep controller}
%
We now derive an estimate for the local error using the two solutions of order $p$ and $\hat{p}$. We approximate the local error behaviour by a simplified model, focusing on the leading terms. Aiming to keep the norm of the local error equal to a desired tolerance, this determines the new timestep. This timestep controller gives us $\Delta t_{n+1}$ based on the previous timestep $\Delta t_n$, the local error estimate and a tolerance $\tau$.

Assume two time-integration schemes $(\mathcal{N}^{t, \Delta t},\,\,\mathcal{N}^{t, \Delta t}_-)$ with orders $(p,\,\,\hat{p})$ and principal error functions $(\bm{\phi},\,\,\bm{\phi}_-)$. Embedded Runge-Kutta schemes \cite{haiwan:93} are a possible choice, as they have the advantage that the embedded solution uses the same stage derivatives, requiring essentially no extra computation.

We use a \textit{local extrapolation} approach to estimate the local error
\begin{equation}\label{EQ LOCAL ERROR LOW}
\bm{\ell}_n
:= \left(\mathcal{N}_-^{t_n, \Delta t_n} - \mathcal{M}^{t_n, \Delta t_n} \right)\bm{u}_n
\end{equation}
by
\begin{equation}\label{EQ LOCAL ERROR LOW EST}
\hat{\bm{\ell}}_n
:= \left(\mathcal{N}_-^{t_n, \Delta t_n} - \mathcal{N}^{t_n, \Delta t_n} \right)\bm{u}_n.
\end{equation}
The leading term of this error estimate, characterized by the principal error function $\bm{\phi}_-$, matches the leading term of the local error \eqref{EQ LOCAL ERROR LOW}. Higher order terms with regards to $\Delta t_n$ will differ. Note that this local error is not the global error increment from \eqref{EQ ERR PROP NUM}, but the one corresponding to $\mathcal{N}^{t, \Delta t}_-$. We model the local error using
\begin{equation}\label{EQ LOCAL ERROR MODEL}
\bm{m}_n
:= \Delta t_n^{\hat{p} + 1} \bm{\phi}_-(t_n, \bm{u}_n),
\end{equation}
assuming $\bm{\phi}_-(t_n, \bm{u}_n)$ to be slowly changing. The next step of this model yields
\begin{equation}\label{EQ ERROR MODEL APPROX}
\bm{m}_{n+1}
\approx \Delta t_{n+1}^{\hat{p} + 1} \bm{\phi}(t_n, \bm{u}_n)
= \left( \frac{\Delta t_{n+1}}{\Delta t_n}\right)^{\hat{p} +1} \bm{m}_n
\approx \left( \frac{\Delta t_{n+1}}{\Delta t_n}\right)^{\hat{p} +1} \hat{\bm{\ell}}_n.
\end{equation}
Aiming for $\|\bm{m}_{n+1}\|= \tau$ gives the well-known deadbeat controller
\begin{equation}\label{EQ CONT DEADBEAT}
\Delta t_{n+1} = \Delta t_n \left( \frac{\tau}{\| \hat{\bm{\ell}}_n \|}\right)^{1/(\hat{p}+1)}.
\end{equation}
%
%
\subsubsection{Convergence in the solution}\label{SEC BASIC CONV}
We now show convergence with $\bm{e}_n = \mathcal{O}(\tau^{p/(\hat{p} + 1)})$ for the adaptive method consisting of a time-integration scheme of order $p$, the error estimate \eqref{EQ LOCAL ERROR LOW EST}, controller \eqref{EQ CONT DEADBEAT} and a suitable initial step-size. 

First we build a relation between global error and maximal timestep with Lemma \ref{LEM CONV}. Corollary \ref{COR BASIC CONV} relaxes this relation to general timesteps in dependence on the tolerance $\tau$. We cannot use the timesteps from controller \eqref{EQ CONT DEADBEAT} directly, since their dependence on $\tau$ is more involved. Instead, we construct a reference timestep series which fulfills the requirements in both Lemma and Corollary and gives the targeted convergence rate. With Theorem \ref{THRM BASIC REF} we show the timesteps from the controller \eqref{EQ CONT DEADBEAT} converge to the reference timesteps for $\tau \rightarrow 0$, which gives convergence with the rate $\mathcal{O}(\tau^{p/(\hat{p} + 1)})$.
\begin{lemma}\label{LEM CONV}
Let problem \eqref{EQ BASE IVP} have a sufficiently smooth right-hand side $\bm{f}$, such that a scheme $\mathcal{N}^{t, \Delta t}$ of order $p$ has the global error increment
\begin{equation*}
(\mathcal{N}^{t_n, \Delta t_n} - \mathcal{M}^{t_n, \Delta t_n})\bm{u}_n = \Delta t^{p + 1}_n \bm{\phi}(t_n, \bm{u}_n) + \mathcal{O}(\Delta t_n^{p + 2}).
\end{equation*}
Assume a mesh $t_0 < \cdots < t_N = t_e$ with timesteps $\Delta t_n = t_{n+1} - t_n$ and the step-size function
\begin{equation*}
\theta : [t_0, t_e] \rightarrow (\theta_{min},1], \quad \theta_{min} > 0,
\end{equation*}
fulfilling
\begin{equation}\label{EQ LEMMA DT}
\Delta t_n = \theta(t_n) \Delta T + \mathcal{O}(\Delta T^{1 + \epsilon}), \quad \epsilon > 0,
\end{equation}
for some $\Delta T > 0$. Then the global error \eqref{EQ GLOBAL ERR} fulfills
\begin{equation*}
\bm{e}_n = \bm{u}_n - \bm{u}(t_n) = \mathcal{O}(\Delta T^p), \quad \forall\, t_n, \quad n = 0, \ldots,\, N,
\end{equation*}
for $\Delta T \rightarrow 0$.
\end{lemma}
\begin{proof}
We first neglect the $\mathcal{O}(\Delta T^{1 + \epsilon})$ term in \eqref{EQ LEMMA DT}. Under these assumptions a proof of $\|\bm{e}_n\| = \mathcal{O}(\Delta T^p)$ can be found in \cite[pp. 68]{gear1971_NUMIVP}.

To extend this result to $\Delta t_n = \theta(t_n) \Delta T + \mathcal{O}(\Delta T^{1 + \epsilon})$, we define
\begin{equation*}
\theta^*(t_n) := \frac{\theta(t_n)}{c} + \mathcal{O}(\Delta T^\epsilon)
\quad \text{and} \quad \Delta T^* := c\, \Delta T,
\end{equation*}
for some $c > 1$. This gives $\Delta t_n = \theta^*(t_n) \Delta T^*$, where
\begin{equation*}
\theta^*(t_n) \leq \frac{1}{c} + \mathcal{O}(\Delta T^\epsilon),
\end{equation*}
which fulfills $0 < \theta^*(t_n) \leq 1$ for $\Delta T$ sufficiently small. This means the general case \eqref{EQ LEMMA DT} is also covered by the proof in \cite[pp. 68]{gear1971_NUMIVP} and for $\Delta T \rightarrow 0$ we get $\bm{e}_n = \mathcal{O}({\Delta T^*}^p) = \mathcal{O}(\Delta T^p)$.\qed
\end{proof}
Using this Lemma, we can now link the global error to the tolerance.
\begin{corollary}\label{COR BASIC CONV}
Assume the smoothness requirements of Lemma \ref{LEM CONV} to be met and assume a scheme of order $p$ to get $\bm{u}_n$. Assume a mesh $t_0 < \cdots < t_N = t_e$ with timesteps $\Delta t_n = t_{n+1} - t_n,\,\, n = 0, \ldots,\, N-1$ that fulfill
\begin{equation*}
\Delta t_n = \mathcal{O}(\tau^{1/q}), \quad \Delta t_n > 0.
\end{equation*}
Then, the global error fulfills
\begin{equation*}
\bm{e}_n = \bm{u}_n - \bm{u}(t_n) = \mathcal{O}(\tau^{p/q}), \quad \forall\, t_n, \quad n = 0, \ldots,\, N,
\end{equation*}
for $\tau \rightarrow 0$.
\end{corollary}
\begin{proof}
The maximal step-size is $\Delta T = \max \{\Delta t_n \,|\, 0 \leq n \leq N-1\}$. The corresponding step-size function is
\begin{equation*}
\theta(t_n) = \frac{\Delta t_n}{\Delta T},
\end{equation*}
which fulfills $0 < \theta(t_n) \leq 1$. We thus meet all assumptions of Lemma \ref{LEM CONV} and get $\bm{e}_n = \bm{u}_n - \bm{u}(t_n) = \mathcal{O}(\Delta T^p) = \mathcal{O}(\tau^{p/q})$.\qed
\end{proof}
We cannot apply Corollary \ref{COR BASIC CONV} to the timesteps \eqref{EQ CONT DEADBEAT} directly, since they have a more complex dependence on $\tau$. Therefore we use reference timesteps $\Delta t_n^{\text{ref}} = \mathcal{O}(\tau^{1/q})$. We show $\Delta t_n \rightarrow \Delta t_n^{\text{ref}}$ for $\tau \rightarrow 0$ with a difference of at most $\mathcal{O}(\Delta T_{\text{ref}}^2)$ and can apply Lemma \ref{LEM CONV}.

We define the reference timesteps
\begin{equation}\label{EQ DT REFERENCE BASIC}
\Delta t_n^{\text{ref}} 
:= \left( \frac{\tau}{c_n \| \bm{\phi}_-(t_n, \bm{u}(t_n))\|} \right)^{1/(\hat{p}+1)},
\end{equation}
where $\bm{\phi}_-$ is the principal error function \eqref{EQ PRINCP ERROR FUNC} corresponding to $\mathcal{N}^{t, \Delta t}_-$ and $c_n$ is given by
\begin{equation*}
c_n = \begin{cases}\mathcal{O}(1), & n = 0,\\ 1, & n > 0,\end{cases}
\end{equation*}
where $\mathcal{O}(1)$ is with respect to $\tau \rightarrow 0$ and $c_0 > 0$. This adds a degree of freedom to choose the initial timestep. For \eqref{EQ DT REFERENCE BASIC} to be well-defined we require $\bm{f}$ in problem \eqref{EQ BASE IVP} to be sufficiently smooth and define
\begin{equation*}
\phi_{-,\min} := \min_{t \in [t_0, t_e]} \| \bm{\phi}_-(t, \bm{u}(t))\|,
\end{equation*}
where we assume $\phi_{-,\min} > 0$. This gives the maximal timestep
\begin{equation}\label{EQ DTMAX BASIC}
\Delta T_{\text{ref}} = \left( \frac{\tau}{\max\{ 1, c_0\} \, \phi_{-,\min}} \right)^{1/(\hat{p}+1)}.
\end{equation}
We have $\Delta t_n^{\text{ref}} \leq \Delta T_{\text{ref}} = \mathcal{O}(\tau^{1/(\hat{p}+1)})$. Applying Lemma \ref{LEM CONV} gives us $\bm{e}_n = \mathcal{O}(\tau^{p/(\hat{p} + 1)})$, for a time-integration scheme of order $p$. We now show convergence of the adaptive method with timesteps from \eqref{EQ CONT DEADBEAT}.
\begin{theorem}\label{THRM BASIC REF}
Let problem \eqref{EQ BASE IVP} have a sufficiently smooth $\bm{f}$. Assume an adaptive method consisting of:
\begin{enumerate}
\item A pair of schemes ($\mathcal{N}^{t, \Delta t},\,\,\mathcal{N}^{t, \Delta t}_-)$ with orders ($p,\,\,\hat{p}$) with $p > \hat{p}$,
\item the error estimator \eqref{EQ LOCAL ERROR LOW EST},
\item the deadbeat controller \eqref{EQ CONT DEADBEAT},
\item an initial timestep $\Delta t_0 = \mathcal{O}(\tau^{1/(\hat{p} + 1)})$.
\end{enumerate}
If the principal error function $\bm{\phi}_-$ to $\mathcal{N}_-^{t, \Delta t}$ fulfills
\begin{equation}\label{EQ REQ PHI MIN}
\min_{t_0 \leq t \leq t_e}{||\bm{\phi}_-(t, \bm{u}(t))||} > 0,
\end{equation}
then the adaptive method is convergent with
\begin{equation*}
\bm{e}_n = \bm{u}_n- \bm{u}(t_n) = \mathcal{O}(\tau^{p/(\hat{p} + 1)}),
\quad \forall t_n, \quad n = 0, \ldots,\, N, \quad \text{and} \quad \tau \rightarrow 0.
\end{equation*}
\end{theorem}
\begin{proof}
By induction we show the timesteps fulfill
\begin{equation*}
\Delta t_n 
= \Delta t_n^{\text{ref}} + \mathcal{O}(\Delta T_{\text{ref}}^2)
= \mathcal{O}(\tau^{1/(\hat{p}+1)}).
\end{equation*}
We choose $c_0$ in $\Delta t_0^{\text{ref}}$ such that $\Delta t_0 = \Delta t_0 ^{\text{ref}}$, meaning the \textit{induction base} is met. The timestep given by the controller is
\begin{equation*}
\Delta t_{n+1} 
= \Delta t_n \left( \frac{\tau}{\| \hat{\bm{\ell}}_n \|}\right)^{1/(\hat{p}+1)}
= \left( \frac{\tau}{\|\bm{\phi}_-(t_n, \bm{u}_n) + \mathcal{O}(\Delta t_n)\|}\right)^{1/(\hat{p}+1)}.
\end{equation*}
We expand the denominator in $\bm{\phi}_-(t_{n+1}, \bm{u}(t_{n+1}))$ and get
\begin{equation*}
\Delta t_{n+1} 
= \left( \frac{\tau}{\|\bm{\phi}_-(t_{n+1}, \bm{u}(t_{n+1})) + \mathcal{O}(\Delta t_n)\|}\right)^{1/(\hat{p}+1)}.
\end{equation*}
We perform another expansion to separate the $\mathcal{O}(\Delta t_n)$ term and get
\begin{equation*}
\Delta t_{n+1} 
= \underbrace{\left( \frac{\tau}{\|\bm{\phi}_-(t_{n+1}, \bm{u}(t_{n+1})) \|}\right)^{1/(\hat{p}+1)}}_{=\Delta t_n^{\text{ref}}} + \mathcal{O}(\tau^{1/(\hat{p}+1)} \Delta t_n).
\end{equation*}
We now consider the $\mathcal{O}$ term. From the definition of the maximal timestep \eqref{EQ DTMAX BASIC} we know $\Delta T_{\text{ref}} = \mathcal{O}(\tau^{1/(\hat{p}+1)})$. The induction hypothesis gives us $\Delta t_n = \Delta t_n^{\text{ref}} + \mathcal{O}(\Delta T^2_{\text{ref}}) \leq \Delta T_{\text{ref}} + \mathcal{O}(\Delta T^2)$. Thus we showed the induction step. By Corollary \ref{COR BASIC CONV} we then get the result $\bm{e}_n = \mathcal{O}(\tau^{p/(\hat{p} + 1)})$.\qed
\end{proof}
Thus we established convergence for the derived adaptive method for a suitable initial timestep $\Delta t_0$. The assumption \eqref{EQ REQ PHI MIN} is a requirement of controllability in the asymptotic regime. The global error would not be controllable by means of local errors, if the local error vanishes at some point. Further we built a structure with which one can prove similar results for different controllers, e.g. PID controllers \cite{Soderlind2003}. To prove convergence one can either show \eqref{EQ LEMMA DT} using suitable reference timesteps or show a maximal deviation of $\mathcal{O}(\Delta T_{\text{ref}}^{1 + \epsilon})$, $\epsilon > 0$ of a given controller from the deadbeat controller \eqref{EQ CONT DEADBEAT}.
%
\section{Goal oriented adaptivity using local error estimates}\label{SEC LOCAL GOAL}
%
We now consider the goal oriented setting \eqref{EQ J DEF} for problem \eqref{EQ BASE IVP} and are only interested in the QoI $J(\bm{u})$. We approximate the integral in $J$ using quadrature and $\bm{u}(t)$ by the numerical solution $\bm{u}_h$ to get
\begin{equation}\label{EQ J DISCR}
J_h(\bm{u}_h) := 
\sum_{n = 0}^{N-1} \Delta t_n \sum_{k=0}^s \sigma_k j(t_n^{(k)}, \bm{u}_n^{(k)})
\approx
\int_{t_0}^{t_e} j(t, \bm{u}(t))dt = J(\bm{u}).
\end{equation}
Here $\bm{u}_n^{(k)} \approx \bm{u}(t_n^{(k)})$ and $\sigma_k$ are the evaluation points resp. weights for the quadrature scheme. We assume an embedded Runge-Kutta scheme for time-integration. 

As we are now only interested in the QoI, we derive an adaptive method that is convergent in the QoI and goal oriented. The method aims to be more efficient by taking into account the QoI for the error estimate. Convergence in the QoI will be shown based on convergence in the solution. With the following Theorem we establish the connection between convergence rates.
\begin{theorem}\label{THRM ORDERS}
Assume $\bm{f}$ in problem \eqref{EQ BASE IVP} and $j$ in the QoI \eqref{EQ J DEF} to be sufficiently smooth, a mesh $t_0 < \cdots < t_N = t_e$ with timesteps $\Delta t_n = t_{n+1} - t_n = \mathcal{O}(\tau^{1/q})$ and an approximation $J_h \approx J$ \eqref{EQ J DISCR} by a quadrature scheme of order $r$. Further assume an approximation $\bm{u}_h \approx \bm{u}(t)$ with order
\begin{equation}\label{EQ J STAGE VALS}
\bm{u}_n^{(k)} - \bm{u}(t_n^{(k)}) = \mathcal{O}(\tau^{p/q})
\end{equation}
for all $n, k$. Then the error in the QoI fulfills
\begin{equation*}
e^J := | J(\bm{u}) - J_h(\bm{u}_h)| = \mathcal{O}(\tau^{\min(r,p)/q}).
\end{equation*}
\end{theorem}
\begin{proof}
By splitting the error, we obtain
\begin{equation}\label{EQ J ERROR SPLIT}
e^J \leq \underbrace{|J(\bm{u}) - J_h(\bm{u})|}_{\text{quadrature error}} + \underbrace{|J_h(\bm{u}) - J_h(\bm{u}_h)|}_{\text{time-integration error}}
\end{equation}
and can deal with the two errors separately.

%
%
An estimate for general numerical quadrature schemes of order $r$ gives
\begin{equation*}
|J(\bm{u}) - J_h(\bm{u})| \leq \left|\sum_{n=0}^{N-1}{c_q {\Delta t_n}^{r+1} \max_{t_n \leq t \leq t_{n+1}}{|(j(t, \bm{u}(t)))^{(r)}|}}\right|,
\end{equation*}
with a constant $c_q$. Using the bound $j_{\max} := \max_{t_0 \leq t \leq t_e}{|(j(t, \bm{u}(t)))^{(r)}|}$, we get
\begin{equation*}
|J(\bm{u}) - J_h(\bm{u})| 
\leq c_q\,j_{\max} \sum_{n=0}^{N-1} {\Delta t_n}^{r+1} 
\leq c_q\,j_{\max} \sum_{n=0}^{N-1} \Delta t_n\, \mathcal{O}(\tau^{r/q})
= \mathcal{O}(\tau^{r/q}).
\end{equation*}
%
%
For the time-integration error we have
\begin{equation*}
|J_h(\bm{u}) - J_h(\bm{u}_h)| \leq \sum_{n=0}^{N-1} \Delta t_n \sum_{k=0}^{s} \left|\sigma_k \left(j(t_n^{(k)}, \bm{u}_n^{(k)}) - j(t_n^{(k)}, \bm{u}(t_n^{(k)}))\right)\right|,
\end{equation*}
where we linearise $j(t_n^{(k)}, \bm{u}(t_n^{(k)}))$ and use assumption \eqref{EQ J STAGE VALS} to get
\begin{equation*}
j(t_n^{(k)}, \bm{u}(t_n^{(k)})) = j(t_n^{(k)}, \bm{u}_n^{(k)}) + \frac{\partial j(t_n^{(k)}, \bm{u}_n^{(k)})}{\partial \bm{u}}(\underbrace{\bm{u}(t_n^{(k)}) - \bm{u}_n^{(k)}}_{ = \, \mathcal{O}(\tau^{p/q})}) + \mathcal{O}((\tau^{p/q})^2).
\end{equation*}
This yields
\begin{equation*}
|J_h(\bm{u}) - J_h(\bm{u_h})| 
\leq \sum_{n=0}^{N-1} \Delta t_n \sum_{k=0}^{s} \mathcal{O}(\tau^{p/q})
= \mathcal{O}(\tau^{p/q}).
\end{equation*}
Summing up quadrature and time-integration error yields
\begin{equation}\label{EQ ORDERS RESULTS}
e^J 
\leq |J(\bm{u}) - J_h(\bm{u})| +|J_h(\bm{u}) - J_h(\bm{u}_h)| 
= \mathcal{O}(\tau^{r/q}) + \mathcal{O}(\tau^{p/q})
= \mathcal{O}(\tau^{\min(r,p)/q}).
\end{equation}
\qed
\end{proof}
Here, we combined statements on convergence in the QoI and the respective rates. The assumption $\Delta t_n = \mathcal{O}(\tau^{1/q})$ gives $\bm{u}_n - \bm{u}(t_n) = \mathcal{O}(\tau^{p/q})$, for all $n = 0,\ldots,\,\, N-1$ by Corollary \ref{COR BASIC CONV}. Using linear interpolation for an intermediate point $t_n^{(k)} \in (t_n, t_{n+1})$, one gets at most $\bm{u}_n^{(k)} - \bm{u}(t_n^{(k)}) = \mathcal{O}(\tau^{2/q})$. The requirement \eqref{EQ J STAGE VALS} becomes relevant for schemes of order $p > 2$ and is discussed in the end of section \ref{SEC LOCAL GOAL CONV}.

Our idea is now to use a goal oriented error estimate to obtain step-sizes more suitable to address the error in the QoI. In practical computations this should lead to a gain in efficiency.

We first derive our error estimate and controller, for the resulting goal oriented adaptive method we show convergence in the QoI with Theorem \ref{THRM J CONV}. We make an analysis to predict the performance in section \ref{SEC ERR PROP}.
%
%
\subsection{Error estimate and timestep controller}
%
In the proof of Theorem \ref{THRM ORDERS} we see two different error sources - time-integration and quadrature, see \eqref{EQ J ERROR SPLIT}. While one can estimate the quadrature error, doing so is not necessary. Using an error estimate based on the time-integration error only, we will get an adaptive method that is convergent in the QoI.

Neglecting the quadrature error we have
\begin{equation*}
e^J 
\approx
\sum_{n=0}^{N-1}\Big|\Delta t_n \sum_{k=0}^s \sigma_k \left(j(t_n^{(k)}, \bm{u}_n^{(k)}) - j(t_n^{(k)}, \bm{u}(t_n^{(k)}))\right)\Big|.
\end{equation*}
As we generally do not have error estimates for the intermediate points of the quadrature scheme, we approximate the above term-wise by the rectangular rule
\begin{equation}\label{EQ J TIME INT ERROR}
e_n^J 
:= \Delta t_n \Big| j(t_{n+1}, \bm{u}_{n+1}) - j(t_{n+1}, \bm{u}(t_{n+1}))\Big|.
\end{equation}
Note that this is an approximation of the time-integration error $J_h(\bm{u}) - J_h(\bm{u}_h)$ and does not place general restrictions on choices for quadrature schemes. The global error propagation form of \eqref{EQ J TIME INT ERROR} is 
\begin{align}
e_{n+1}^J 
= & \,\Delta t_n \left( 
j\left(t_{n+1}, \mathcal{N}^{t_n, \Delta t_n}\bm{u}_n\right) - 
j\left(t_{n+1}, \mathcal{M}^{t_n, \Delta t_n}\bm{u}(t_n)\right)\right)
\label{EQ J ERROR PROP 1}\\
= & \, \underbrace{\Delta t_n \left(
j\left(t_{n+1}, \mathcal{N}^{t_n, \Delta t_n}\bm{u}_n\right) - 
j\left(t_{n+1}, \mathcal{M}^{t_n, \Delta t_n}\bm{u}_n\right)\right)
}_{\text{global error increment}} + \label{EQ J SIMPL ERR INCR}\\
& \, \underbrace{\Delta t_n \left(
j\left(t_{n+1}, \mathcal{M}^{t_n, \Delta t_n}\bm{u}_n\right) - 
j\left(t_{n+1}, \mathcal{M}^{t_n, \Delta t_n}\bm{u}(t_n)\right)\right)
}_{\text{global error propagation}}. \label{EQ J ERROR PROP 3}
\end{align}
Again, we do not know the global error propagation dynamics, but we can estimate the global error increment and control it using timesteps. We use local extrapolation with a scheme $\mathcal{N}^{t, \Delta t}_-$ of order $\hat{p} < p$ and control
\begin{equation}\label{EQ J LOCAL ERROR LOW}
\Delta t_n \ell_n^j := 
\Delta t_n \left(
j\left(t_{n+1}, \mathcal{N}_-^{t_n, \Delta t_n}\bm{u}_n\right) - 
j\left(t_{n+1}, \mathcal{M}^{t_n, \Delta t_n}\bm{u}_n\right)\right).
\end{equation}
This is the global error increment \eqref{EQ J SIMPL ERR INCR}, but corresponding to $\mathcal{N}^{t_n, \Delta t_n}_-$. We estimate \eqref{EQ J LOCAL ERROR LOW} by
\begin{equation}\label{EQ J ERROR EST}
\Delta t_n \hat{\ell}_n^j := 
\Delta t_n \left(
j\left(t_{n+1}, \mathcal{N}^{t_n, \Delta t_n}_-\bm{u}_n\right) - 
j\left(t_{n+1}, \mathcal{N}^{t_n, \Delta t_n}\bm{u}_n\right)\right).
\end{equation}
To construct a controller we need a model for \eqref{EQ J LOCAL ERROR LOW}. As $j$ may be non-linear, we linearise in $\bm{u}$ and get
\begin{align}
\ell_n^j = \,\, & j\left(t_{n+1}, \mathcal{N}^{t_n, \Delta t_n}_-\bm{u}_n\right) - j\left(t_{n+1}, \mathcal{M}^{t_n, \Delta t_n}\bm{u}_n\right)\nonumber\\
= \,\, &
\frac{\partial j\left(t_{n+1}, \mathcal{M}^{t_n, \Delta t_n}\bm{u}_n\right)}{\partial \bm{u}} 
\underbrace{\left(\mathcal{N}^{t_n, \Delta t_n}_- - \mathcal{M}^{t_n, \Delta t_n} \right) \bm{u}_n}_{\stackrel{\eqref{EQ LOCAL ERROR LOW}}{=} \bm{\ell}_n} + \mathcal{O}(\Delta t_n^{2\hat{p} + 2}). \label{EQ J AUX LIN}
\end{align}
As model we choose the leading term of \eqref{EQ J AUX LIN}
\begin{equation*}
m_n^j
:= \Delta t_n \frac{\partial j(t_{n+1}, \bm{u}_{n+1})}{\partial \bm{u}} \bm{m}_n
\end{equation*}
and assume the derivative term to be slowly changing. Here, $\bm{m}_n$ is the classical error estimate \eqref{EQ LOCAL ERROR MODEL}. For this model the next step yields
\begin{align*}
m_{n+1}^j
& \approx \Delta t_{n+1} \frac{\partial j(t_{n+1}, \bm{u}_{n+1})}{\partial \bm{u}} \bm{m}_{n+1}
= \frac{\Delta t_{n+1}^{\hat{p}+2}}{\Delta t_n^{\hat{p} + 1}}\frac{\partial j(t_{n+1}, \bm{u}_{n+1})}{\partial \bm{u}} \bm{m}_n\\
& \stackrel{\eqref{EQ ERROR MODEL APPROX}}{\approx} \frac{\Delta t_{n+1}^{\hat{p}+2}}{\Delta t_n^{\hat{p} + 1}}\,\,\frac{\partial j(t_{n+1}, \bm{u}_{n+1})}{\partial \bm{u}}\, \hat{\bm{\ell}}_n
\stackrel{\eqref{EQ J ERROR EST}\,-\,\eqref{EQ J AUX LIN}}{\approx} \Delta t_{n+1} \left( \frac{\Delta t_{n+1}}{\Delta t_n}\right)^{\hat{p} + 1} \hat{\ell}^j_n.
\end{align*}
We aim to control the error \textit{per unit interval}, per step, which means aiming for $|m_{n+1}^j| = \Delta t_{n+1}\, \tau$. This is not to be confused with the common \textit{Error Per Unit Step} (EPUS) approach in classical timestep control. We get the deadbeat controller
\begin{equation}\label{EQ J DEADBEAT CONT}
\Delta t_{n+1}
= \Delta t_n \left( \frac{\tau}{|\hat{\ell}^j_n|}\right)^{1/(\hat{p}+1)}.
\end{equation}
We thus constructed a timestep controller to control the error in the QoI \eqref{EQ J DEF} using only local error estimates in $j(t, \bm{u})$. In the next section we show that the resulting adaptive method is convergent in the solution and QoI.

Comparing the implementation of this adaptive method to the classical one from section \ref{SEC LOCAL ERR BASIC}, we only require the density function $j$. This we need regardless of the used method, as it is necessary for the evaluation of $J(\bm{u})$.
%
\subsection{Convergence in the quantity of interest}\label{SEC LOCAL GOAL CONV}
%
We now show convergence of the derived goal oriented adaptive scheme in the QoI, using Theorem \ref{THRM ORDERS}. While we use the same controller, we have a different error estimator and cannot use Corollary \ref{COR BASIC CONV} directly. This is due to the timesteps \eqref{EQ J DEADBEAT CONT} converging to a different series of reference timesteps. We define these and repeat the steps of Theorem \ref{THRM BASIC REF}, showing convergence of the steps from the controller \eqref{EQ J DEADBEAT CONT} to our reference. We use
\begin{equation}\label{EQ J DT REF}
\Delta t_n^{\text{ref}} :=
\left( \frac{\tau}{c_n\,\left|\frac{\partial j(t_{n}, \bm{u}(t_n))}{\partial \bm{u}} \,\bm{\phi}_-(t_n, \bm{u}(t_n))\right|} \right)^{1/(\hat{p}+1)},
\end{equation}
with 
\begin{equation*}
c_n = \begin{cases}\mathcal{O}(1),& n = 0, \\ 1, & n > 0,\end{cases}
\end{equation*}
where $\mathcal{O}(1)$ is for $\tau \rightarrow 0$ and $c_0 > 0$. This gives a degree of freedom in choosing $\Delta t_0$. For the timesteps to be well-defined we require
\begin{equation*}
\phi_{-, \min}^j := \min_{t \in [t_0, t_e]} \left|\frac{\partial j(t, \bm{u}(t))}{\partial \bm{u}} \, \bm{\phi}_-(t, \bm{u}(t))\right| > 0,
\end{equation*}
yielding the maximal timestep
\begin{equation}\label{EQ J DTMAX}
\Delta T_{\text{ref}} := \left( \frac{\tau}{\max\{ 1, c_0\}\,\phi_{-,\min}^j} \right)^{1/(\hat{p}+1)}.
\end{equation}
We have $\Delta t_n^{\text{ref}} \leq \Delta T_{\text{ref}} = \mathcal{O}(\tau^{1/(\hat{p}+1)})$. With the following Theorem we show convergence of the timesteps from controller \eqref{EQ J DEADBEAT CONT} with error estimate \eqref{EQ J ERROR EST} to the reference timesteps \eqref{EQ J DT REF}.
\begin{theorem}\label{THRM J CONV}
Let $\bm{f}$ in \eqref{EQ BASE IVP} and $j$ in \eqref{EQ J DEF} be sufficiently smooth. Assume an adaptive method consisting of:
\begin{enumerate}
\item A pair of schemes $\mathcal{N}^{t, \Delta t}, \mathcal{N}^{t, \Delta t}_-$ with orders $p,\,\,\hat{p}$ and $p > \hat{p}$,
\item a quadrature scheme of order $r$ to approximate $J(\bm{u})$ as in \eqref{EQ J DISCR},
\item schemes $\mathcal{N}^{t, \Delta t}_{(k)}$ to obtain solutions of order $p - 1$ for all quadrature evaluation points, that are not part of the resulting grid,
\item the error estimator \eqref{EQ J ERROR EST},
\item the deadbeat controller \eqref{EQ J DEADBEAT CONT},
\item an initial timestep $\Delta t_0 = \mathcal{O}(\tau^{1/(\hat{p} + 1)})$.
\end{enumerate}
If the principal error function $\bm{\phi}_-$ to $\mathcal{N}_-^{t, \Delta t}$ fulfills
\begin{equation}\label{EQ J PHIMIN}
\min_{t \in [t_0, t_e]} \left|\frac{\partial j(t, \bm{u}(t))}{\partial \bm{u}} \, \bm{\phi}_-(t, \bm{u}(t))\right| > 0,
\end{equation}
then
\begin{equation*}
e^J := |J(\bm{u}) - J_h(\bm{u}_h)| = \mathcal{O}(\tau^{p/(\hat{p}+1)}), \quad \text{for} \quad \tau \rightarrow 0.
\end{equation*}
\end{theorem}
\begin{proof}
We first show convergence in the solution by inductively showing that every step given by the controller \eqref{EQ J DEADBEAT CONT} fulfills
\begin{equation}\label{EQ J DT THEOREM}
\Delta t_n = \Delta t_n^{\text{ref}} + \mathcal{O}(\Delta T^2_{\text{ref}}).
\end{equation}
We choose $c_0$ for $\Delta t_0^{\text{ref}}$ such that $\Delta t_0 = \Delta t_0^{\text{ref}}$. Thus the \textit{induction base} $\Delta t_0 = \Delta t_0^{\text{ref}} + \mathcal{O}(\Delta T^2_{\text{ref}})$ is met. In the controller \eqref{EQ J DEADBEAT CONT} we have the denominator
\begin{align}
\hat{\ell}_n^j 
& = j\left(t_{n+1}, \mathcal{N}_-^{t_n, \Delta t_n} \bm{u}_n\right)
- j\left(t_{n+1}, \mathcal{N}^{t_n, \Delta t_n} \bm{u}_n\right) \nonumber\\
& = \frac{\partial j(t_{n+1}, \bm{u}_{n+1})}{\partial \bm{u}} \hat{\bm{\ell}}_n + \mathcal{O}(\Delta t_n^{2\hat{p} + 2}).\label{EQ J THRM LINEAR}
\end{align}
Repeating the same expansions as in the proof for Theorem \ref{THRM BASIC REF} we have
\begin{equation*}
\hat{\bm{\ell}}_n 
= \Delta t_n^{\hat{p} + 1} \left( \bm{\phi}_-(t_{n+1}, \bm{u}(t_{n+1})) + \mathcal{O}(\Delta t_n)\right)
\end{equation*}
and similarly
\begin{equation*}
\frac{\partial j(t_{n+1}, \bm{u}_{n+1})}{\partial \bm{u}}
= \frac{\partial j(t_{n+1}, \bm{u}(t_{n+1}))}{\partial \bm{u}} + \mathcal{O}(\bm{e}_{n+1}).
\end{equation*}
Here we have $\mathcal{O}(\bm{e}_{n+1}) = \mathcal{O}(\Delta T_{\text{ref}}^p)$ by Corollary \ref{COR BASIC CONV}, as we assume all timesteps leading up to $\bm{e}_{n+1}$ to fulfill \eqref{EQ J DT THEOREM}. Using these approximations in \eqref{EQ J THRM LINEAR} we get
\begin{equation*}
\hat{\ell}_n^j = \Delta t_n^{\hat{p} + 1} \left( \frac{\partial j(t_{n+1}, \bm{u}(t_{n+1}))}{\partial \bm{u}} \bm{\phi}_-(t_{n+1}, \bm{u}(t_{n+1})) + \mathcal{O}(\Delta t_n) + \mathcal{O}(\Delta T_{\text{ref}}^p)\right),
\end{equation*}
which we insert into the controller \eqref{EQ J DEADBEAT CONT} to get
\begin{equation*}
\Delta t_{n+1} = \left( \frac{\tau}{\frac{\partial j(t_{n+1}, \bm{u}(t_{n+1}))}{\partial \bm{u}} \bm{\phi}_-(t_{n+1}, \bm{u}(t_{n+1})) + \mathcal{O}(\Delta t_n) + \mathcal{O}(\Delta T_{\text{ref}}^p)}\right)^{1/(\hat{p}+1)}.
\end{equation*}
Here we can pull out $\mathcal{O}(\Delta t)$ terms and use the induction hypothesis to get
\begin{align*}
\Delta t_{n+1} 
& \stackrel{\eqref{EQ J DT REF}}{=} \Delta t_{n+1}^{\text{ref}} + \tau^{1/(\hat{p}+1)} \left( \mathcal{O}(\Delta t_n) + \mathcal{O}(\Delta T_{\text{ref}}^p)\right)\\
& \stackrel{\eqref{EQ J DTMAX},\,\eqref{EQ J DT THEOREM}}{=} \Delta t_{n+1}^{\text{ref}} + \mathcal{O}(\Delta T_{\text{ref}}^2),
\end{align*}
which shows the induction step holds, yielding $\Delta t_{n+1} = \mathcal{O}(\tau^{1/(\hat{p}+1)})$. We thus proved the induction and get $\Delta t_n = \mathcal{O}(\tau^{1/(\hat{p}+1)})$ for all $n$. This gives us the assumption on step-sizes as needed by Theorem 3.1 and convergence in the solution in the grid-points by Corollary \ref{COR BASIC CONV}, with a rate of $\mathcal{O}(\tau^{p/(\hat{p} + 1)})$.

Now, we want to show that we fulfill assumption \eqref{EQ J STAGE VALS} in Theorem \ref{THRM ORDERS}. Taking a single step of size $\Delta t_n^{(k)}$ from $t_n$, to the quadrature evaluation point $t_n^{(k)} \in (t_n, t_{n+1})$, with the scheme $\mathcal{N}_{(k)}^{t_n, \Delta t_n^{(k)}}$, gives the error
\begin{align*}
\bm{e}_n^{(k)} = \bm{u}(t_n^{(k)}) - \bm{u}_n^{(k)} & 
= \mathcal{N}^{t_n, \Delta t_n^{(k)}}_{(k)}\bm{u}_n - M^{t_n, \Delta t_n^{(k)}} \bm{u}(t_n)\\
& = \underbrace{M^{t_n, \Delta t_n^{(k)}}}_{=\,\,\mathcal{O}(1)} \underbrace{\bm{e}_n}_{=\,\,\mathcal{O}(\tau^{p/(\hat{p} + 1)})} + \underbrace{\left(\mathcal{N}^{t_n, \Delta t_N^{(k)}}_{(k)} - M^{t_n, \Delta t_n^{(k)}} \right)\bm{u}_n}_{=\,\,\mathcal{O}\left(\left(\Delta t_n^{(k)}\right)^p\right)}.
\end{align*}
Here we have $\Delta t_n^{(k)} \leq \Delta t_n = \Delta t_n^{\text{ref}} + \mathcal{O}(\Delta T_{\text{ref}}^2) = \mathcal{O}(\tau^{p/(\hat{p}+1)})$. Hence we have $\bm{e}_n^{(k)} = \mathcal{O}(\tau^{p/(\hat{p} + 1)})$ and fulfill assumption \eqref{EQ J STAGE VALS} of Theorem \ref{THRM ORDERS} for $q = \hat{p} + 1$, which gives us convergence in the QoI with $e^J = \mathcal{O}(\tau^{\min(r,\,p)/(\hat{p} + 1)})$.\qed
\end{proof}
Thus our adaptive method is convergent in the QoI. The requirement of $\phi^j_{-, \min} > 0$ is a requirement on controllability of the global error by means of the local error \eqref{EQ J LOCAL ERROR LOW}. For a $j$ which is linear in $\bm{u}$ this is equivalent to the local error not being in the nullspace of $j(t, \cdot)$. Possible consequences of it being in the nullspace of $j(t, u)$ are shown in the following example. A more general analysis of this is subject of the next section.
\begin{example}
In \cite{Turek1999} the author describes using the lift-coefficient of the flow around a cylinder as density function $j(t, \bm{u})$, which changes sign over time, implying it has zeros. It is observed that large timesteps are chosen when the lift-coefficient is close to zero, leading to "catastrophical results". By the description of this example we can see criterion \eqref{EQ J PHIMIN} not being fulfilled and thus not guaranteeing convergence in the QoI.
\end{example}
We now discuss the time-integration schemes for the quadrature evaluation points, as needed in the assumptions of Theorem \ref{THRM J CONV}. Here, we only need a solution for the points, that are not part of the grid. This is relevant only for quadrature schemes of order $r > 2$, for $r = 2 $ there is the trapezoidal rule. One can also use linear interpolation of the solution at grid points, which can be formally expressed using a combination of the identity operator and $\mathcal{N}^{t_n, \Delta t_n}$, using suitable weights. Linear interpolation will, however, yield at most $\bm{e}_n^{(k)} = \mathcal{O}(\tau^{2/(\hat{p} + 1)})$.

We instead want to use RK schemes and use the already calculated stage derivatives. To determine weights of RK schemes for intermediate points $t_n^{(k)} = t_n + \gamma_k \,\Delta t_n$ for $\gamma_k \in (0, 1]$ one has to modify the RK order conditions the following way:

Taking the order conditions for order $p$, e.g. $\sum_s b_s c_s = \frac{1}{2}$ for $p=2$, one has to multiply the right-hand side by $\gamma_k^p$. This becomes clear when looking into the details of a proof on order conditions \cite[pp. 142]{haiwan:93}. 
\begin{example}
Assume the classic $4$th order Runge-Kutta scheme for time-integration. Since the convergence rate in the QoI \eqref{EQ ORDERS RESULTS} is determined by the minimum order of quadrature and time-integration scheme, we pick the Simpson rule ($r = 4$) for quadrature. To get a $4$th order ($3$rd order local) solution for the point $t_n + \Delta t_n/2$ one can use the RK weights $b^* = \frac{1}{24}(5, 4, 4, -1)$ for the same stage derivatives.
\end{example}
%
\subsection{The nullspace of $j(t, \bm{u})$ and global error propagation}\label{SEC ERR PROP}
%
Convergence of a method is a statement for the limit $\tau \rightarrow 0$ and does not regard global error dynamics. Here, we analyse them to make a qualitative statement about the grid obtained from the goal oriented adaptive method for $\tau > 0$ not in the limit. We establish guidelines to predict grid quality and thus performance of the goal-oriented adaptive method.

We assume a QoI with a density function that is linear in $\bm{u}$,
\begin{equation*}
j(t, \bm{u}(t)) = \bm{w}^T \bm{u}(t), \quad w_i \geq 0.
\end{equation*}
Considering the split \eqref{EQ J ERROR SPLIT} of $e^J$, the quadrature error does not involve the numerical solution. Global error propagation only appears in the time-integration part of the error, which in this case is given by
\begin{equation*}
|J_h(\bm{u}) - J_h(\bm{u}_h)| 
\leq \sum_{n=0}^{N-1} \left| \Delta t_n \sum_{k=0}^s \sigma_k \bm{w}^T \bm{e}_n^{(k)}\right|.
\end{equation*}
We define a weighted seminorm $\| \cdot \|_w: \mathbb{R}^d \mapsto \mathbb{R}$ by 
\begin{equation}\label{EQ SEMINORM DEF}
\| \bm{x} \|_w := \sum_{i=1}^d w_i |x_i|, \quad w_i \geq 0.
\end{equation}
This is a seminorm, since it may have a non-trivial nullspace if $w_i = 0$ for some indices $i$. Throughout this section, we assume $\| \cdot\|_w$ to have a non-trivial nullspace.

Using \eqref{EQ SEMINORM DEF} we get the bound
\begin{equation*}
|J_h(\bm{u}) - J_h(\bm{u}_h)| 
\leq \sum_{n=0}^{N-1} \Delta t_n \sum_{k=0}^s \sigma_k \|\bm{e}_n^{(k)}\|_w
\end{equation*}
and we need to further investigate how $\|\bm{e}_n^{(k)}\|_w$ is affected by global error propagation. Starting from \eqref{EQ J ERROR PROP 1} - \eqref{EQ J ERROR PROP 3}, we extract the error corresponding to a single quadrature evaluation point and get 
\begin{align*}
j(t_{n+1}, \bm{e}_{n+1}^{(k)})
&= j\left(t_{n+1}, \mathcal{M}_{(k)}^{t_n, \Delta t_n} \bm{u}_n\right) - j\left(t_{n+1}, \mathcal{M}^{t_n, \Delta t_n}_{(k)}\bm{u}(t_n)\right) \\
& \quad + j\left(t_{n+1}, \mathcal{N}_{(k)}^{t_n, \Delta t_n} \bm{u}_n\right) - j\left(t_{n+1}, \mathcal{M}^{t_n, \Delta t_n}_{(k)}\bm{u}_n\right).
\end{align*}
Replacing $j$ by $\|\cdot\|_w$ yields
\begin{align}
\|\bm{e}_{n+1}^{(k)} \|_w 
& \leq \underbrace{\left\|\mathcal{M}_{(k)}^{t_n, \Delta t_n} \bm{u}_n - \mathcal{M}^{t_n, \Delta t_n}_{(k)}\bm{u}(t_n) \right\|_w}_{\text{global error propagation}} \label{EQ PROP ERR 1}\\
& \quad + \underbrace{\left\|\left( \mathcal{N}_{(k)}^{t_n, \Delta t_n} - \mathcal{M}^{t_n, \Delta t_n}_{(k)}\right)\bm{u}_n \right\|_w}_{\text{global error increment}}. \label{EQ PROP ERR 2}
\end{align}
We now want to find a bound for the global error propagation term \eqref{EQ PROP ERR 1} depending on $\bm{e}_n$. For this we use Lipschitz-conditions.

Assume a map $\bm{f}: U \rightarrow W$ fulfills the Lipschitz condition
\begin{equation*}
\| \bm{f}(\bm{u}) - \bm{f}(\bm{v}) \|_W \leq \mathcal{L} \|\bm{u} - \bm{v} \|_U
\end{equation*}
with some constant $\mathcal{L}$ and suitable norms for the spaces $U$ and $W$. The \textit{Lipschitz-norm} of $\bm{f}$ is the minimal $\mathcal{L}$ fulfilling the Lipschitz condition, cf. \cite{Soderlind2006a}. We can define an according \textit{Lipschitz-seminorm} in the following way.
\begin{definition}
Assume a map $\bm{f} : U \rightarrow W$ with $\| \cdot \|_w$ being a \textit{seminorm} on $W$ and $\| \cdot \|_U$ being a \textit{norm} on $U$. We define the \textit{Lipschitz-seminorm} (with respect to a given norm on $U$) by
\begin{equation*}
\mathcal{L}_w[\bm{f}] := \sup_{\bm{u} \neq \bm{v}} \frac{\| \bm{f}(\bm{u}) - \bm{f}(\bm{v})\|_w}{\|\bm{u} - \bm{v}\|_U},
\end{equation*}
where $\bm{u}, \bm{v} \in U$.
\end{definition}
Due to the possible non-trivial nullspace, we yet require a norm in the denominator. We can use this definition to bound the global error propagation in \eqref{EQ PROP ERR 1} - \eqref{EQ PROP ERR 2} by
\begin{equation*}
\|\bm{e}_{n+1}^{(k)} \|_w 
\leq \mathcal{L}_w[\mathcal{M}_{(k)}^{t_n, \Delta t_n}]\, \| \bm{e}_n \|_U + \left\|\left( \mathcal{N}_{(k)}^{t_n, \Delta t_n} - \mathcal{M}^{t_n, \Delta t_n}_{(k)}\right)\bm{u}_n \right\|_w,
\end{equation*}
where $\| \cdot \|_U$ is a norm. It will be non-zero in the nullspace of $\| \cdot \|_w$. For a non-trivial nullspace this can be problematic, which we first illustrate by a simple test case and later using numerical test problems in section \ref{SEC NUM RESULTS}.

To get a better idea of the dynamics, we look at the linear case $\bm{f}(\bm{u}) = \bm{A}\bm{u}$ and take $\| \cdot \|_U$ to be the 1-norm.
\begin{lemma}
The Lipschitz-seminorm of a linear operator $\bm{A} \in \mathbb{R}^{n \times m}$, with respect to the 1-norm, is given by
\begin{equation}\label{EQ DEF LIP SEMINORM LIN}
\| \bm{A} \|_w 
= \sup_{\| \bm{x}\|_1 = 1} \| \bm{A} \bm{x}\|_w 
= \max_{j=1,\ldots,\,n}\left( \sum_{i=1}^n w_i |a_{ij}|\right).
\end{equation}
\end{lemma}
\begin{proof}
The first form using the supremum is obtained by defining $\bm{x} := \bm{u} - \bm{v}$ and scaling $\| \bm{x} \|_1 \neq 0$ to $\| \bm{x} \|_1 = 1$ using the homogeneity of the seminorm. For the second part we have
\begin{align*}
\| \bm{A} \bm{x}\|_w & = 
\sum_{i = 1}^n w_i \sum_{j=1}^m | a_{ij} x_j|
= \sum_{i = 1}^n w_i \sum_{j=1}^m | a_{ij}| | x_j| \\ 
& = \sum_{j = 1}^m |x_j| \sum_{i=1}^n w_i | a_{ij}|
\leq \sum_{j = 1}^m |x_j| \left(\max_{j = 1,\ldots,\, m} \sum_{i=1}^n w_i | a_{ij}|\right).
\end{align*}
Using the supremum over all $ \|\bm{x}\|_1 = 1$ yields the result.\qed
\end{proof}
Thus the Lipschitz-seminorm of a linear operator is a (weighted) column max-semi\-norm of the given matrix.
\begin{example}
Consider
\begin{equation*}
\bm{A} = \begin{pmatrix} 2 & 1 \\ 0 & 4\end{pmatrix},
\quad \bm{x} = \begin{pmatrix} 1 \\ 2\end{pmatrix},
\quad \bm{w} = \begin{pmatrix} 1 \\ 0\end{pmatrix}
\Rightarrow
\bm{Ax} = \begin{pmatrix} 4 \\ 8\end{pmatrix}.
\end{equation*}
We have $\| \bm{A}\|_w = 2, \| \bm{x}\|_w = 1, \|\bm{x}\|_1 = 3$ and $\| \bm{Ax}\|_w = 4$. From \eqref{EQ DEF LIP SEMINORM LIN} we get the inequality
\begin{equation*}
\| \bm{A} \bm{x}\|_w \leq \| \bm{A} \|_w \|\bm{x}\|_1.
\end{equation*}
Here we fulfill $\| \bm{A}\bm{x}\|_w \leq \|\bm{A}\|_w \| \bm{x}\|_1$, but have $\| \bm{A}\bm{x}\|_w \nleq \|\bm{A}\|_w \| \bm{x}\|_w$. This shows the inequality $\| \bm{A}\bm{x}\| \leq \|\bm{A} \| \|\bm{x}\|$ for norms does not hold for seminorms. 
\end{example}
Consider the flow map and weights
\begin{equation*}
\bm{M} = \begin{pmatrix} m_{11} & m_{12} \\ m_{21} & m_{22} \end{pmatrix}
, \quad
w = \begin{pmatrix}1\\0\end{pmatrix}.
\end{equation*}
The diagonal entries of $\bm{M}$ describe dampening or amplification in the nullspace or image and the off-diagonal entries describe (scaled) transport from the nullspace into the image or vice versa. This can also be formulated in an analogous blockwise formulation, then the diagonal blocks $m_{11}$ and $m_{22}$ may include transport inside the nullspace resp. image.

Dampening is generally favorable and amplification may be unavoidable if it is part of the ODE/PDE. Transport from the image into the nullspace is unproblematic, but transport from the nullspace to the image can be highly problematic. In this example the corresponding component is $m_{21}$.

We choose timesteps to control the global error increments \eqref{EQ PROP ERR 2} in the image. Controlling the timesteps we try to keep the seminorm of these increments below a given tolerance. We do not control the error increments in the nullspace, which can be problematic if errors from the nullspace are transported into the image. 

As simulating a process results in errors regardless of the step-size, this is a question of sufficiently resolving relevant processes. Assume a process in the nullspace is faster than a process in the image at a given time. The timesteps are chosen to sufficiently resolve the slow process in the image. The process in the nullspace remains under-resolved, its error exceeding the tolerance. \textbf{If} this error is then transported into the image, the performance of the goal oriented adaptive method suffers.

Likewise the goal oriented controller performs well, if all processes whose errors end up in the image, are sufficiently resolved. This can be due to the image containing the processes which require smaller timesteps. The other case is that processes in the nullspace remain under-resolved but have neglectable impact on $J(\bm{u})$. This may be due to strong damping in $m_{22}$ or lack of transport with $m_{12}$ being small.

Due to potentially complicated dynamics of the system, it is hard to clearly identify which processes are neglectable.
\begin{example}
In \cite{Wick2017} the author simulates flow-driven fracturing of an obstacle. The QoI $j(t, \bm{u})$ is the displacement of the obstacle in flow direction, measured at the tip of the outflow edge. Using this density function to control timesteps, there may be a small delay from the flow building up around the obstacle and the displacement occurring. This delay would result in choosing too large timesteps when the displacement is just starting to grow, but the flow pattern around the obstacle is already beginning to form. The author observes significant error reductions after a certain tolerance, which may be the point at which the inflow is sufficiently resolved.
\end{example}
%
\section{Numerical Results}\label{SEC NUM RESULTS}
%
We now test the results of Theorem \ref{THRM ORDERS} on convergence rates numerically. Further we compare performance of the DWR method and the local error based adaptive methods. Verification of the results of section \ref{SEC LOCAL ERR BASIC} are not presented as they are well established.

The experiments were run on a Intel i7-3930K 3.20 GHz CPU and implemented in Python 2.7 using FEniCS \cite{logg2012_FENICS}. 

The code is available at \url{http://www.maths.lu.se/fileadmin/maths/personal_staff/PeterMeisrimel/goal_oriented_time.zip}.

The following specifications are shared for all test-cases. For local error based methods using timestep-controllers, we bound the rate by which timesteps change \cite{haiwan:93} by 
\begin{equation}\label{EQ DT LIMITER}
\Delta t_{n+1} = \Delta t \,\min (f_{\max}, \max (f_{\min}, \text{ind})),
\quad \text{ind} = \left( \frac{\tau}{\tilde{\ell}_n} \right)^{1/(\hat{p}+1)}.
\end{equation}
Here $\tilde{\ell} = \| \hat{\ell}_n\|$ for \eqref{EQ CONT DEADBEAT} resp. $\tilde{\ell}_n = |\hat{\ell}_n^j|$ for \eqref{EQ J DEADBEAT CONT}. The purpose of this is to provide more computational stability by preventing too large or too small timestep changes. In practice this will not take effect for $\tau \rightarrow 0$. We use $f_{\max} = 3$ and $f_{\min} = 0.01$ and do not reject timesteps. For the initial timestep we use $\Delta t_0 = \tau^{1/(\hat{p}+1)}$.

For the DWR method we use an initial grid with $10$ equidistant cells. As refinement strategy we use fixed-rate refinement \cite{bangerth2013adaptive} with $X = 0.8$ and $Y = 0$. This means we refine 80$\%$ of cells corresponding to the largest errors, where refinement means to split the cell into two equally sized cells. To approximate $z \approx z_h^+$ we use a finer grid, dividing all time-intervals in half.

We refer to the adaptive method from section \ref{SEC LOCAL ERR BASIC} as the "Classic" method and to the one from section \ref{SEC LOCAL GOAL} as the "Goal oriented" method.
%
\subsection{Test problem}
%
As a simple test problem with known global error dynamics we consider
\begin{equation}\label{EQ TOY MAIN}
\dot{\bm{u}}(t) =
\begin{pmatrix}
-1 & 1 \\ 0 & k
\end{pmatrix}
\bm{u}(t)
,\quad \bm{u}(t_0) = \bm{u}_0 = 
\begin{pmatrix}
1 \\ 1
\end{pmatrix}
,\quad t \in [t_0, t_e].
\end{equation}
We use $[t_0, t_e] = [0,2]$ and vary the stiffness by $k < 0$.
\subsubsection*{DWR estimate}
The unique solution to \eqref{EQ TOY MAIN} is in $\mathcal{C}^{\infty}([t_0, t_e]) \times \mathcal{C}^{\infty}([t_0, t_e])$. We define the finite element space $V_h := \{ u \in \mathcal{C}([t_0, t_e]): u\big|_{I_n} \in \mathcal{P}^q(I_n)\}$ denoting the space of continuous piece-wise $q$-th order polynomials, where $I_n = [t_n, t_{n+1}]$. Using test-functions $\bm{\phi}_h := (\phi_1, \phi_2) \in V_h \times V_h$ and $\bm{u}_h := (u_1, u_2) \in V_h \times V_h$ we have a weak formulation
\begin{equation*}
\int_{t_0}^{t_e} \left( \dot{u}_1 + u_1 - u_2\right)\phi_1 + \left( \dot{u}_2 - k \, u_2\right)\phi_2 \, dt = 0.
\end{equation*}
Using $(\cdot, \cdot)_{I_n}$ as the standard $L^2$ scalar product over $I_n$, we have
\begin{equation*}
\sum_{n = 0}^{N-1} \left( \dot{u}_1 + u_1 - u_2, \phi_1\right)_{I_n} + \left( \dot{u}_2 - k\, u_2, \phi_2\right)_{I_n} = 0,
\end{equation*}
where the entire left-hand side defines the bilinear form $A(\bm{u}_h, \bm{\phi}_h)$. Let $\bm{z}$ be the exact adjoint solution and $\bm{z}_h = (z_1, z_2)$ its finite element approximation, we get $e^J = A(\bm{u}_h, \bm{z} - \bm{z}_h)$. We approximate $\bm{z}$ by $\bm{z}^+_h = (z_1^+, z_2^+)$ to get
\begin{align*}
e^J \approx &\,\, A(\bm{u}_h, \bm{z}^+_h - \bm{z}_h) \\
\lesssim & \sum_{n=0}^{N-1} \left| \int_{t_n}^{t_{n+1}} \left(\dot{u}_1 + u_1 - u_2\right)\left( z^+_1 - z_1\right) + \left( \dot{u}_2 + k\, u_2\right)\left( z^+_2 - z_2\right)dt \right|
\end{align*}
Defining 
\begin{align*}
R_1(t) & := \left(\dot{u}_1 + u_1 - u_2\right)\left( z^+_1 - z_1\right),\\
R_2(t) &:= \left( \dot{u}_2 + k\, u_2\right)\left( z^+_2 - z_2\right),
\end{align*}
we get the final error estimate $\eta_h(\bm{u}_h)$ using the composite trapezoidal rule
\begin{align*}
\eta(\bm{u}_h) & := \sum_{n=0}^{N-1} \frac{\Delta t_n}{4} \left| \sum_{i=1}^2 R_i(t_n) + 2 R_i(t_n + \Delta t_n/2) + R_i(t_{n+1}) \right|.
\end{align*}
%
%
%
\subsubsection{Numerical verification of Theorem \ref{THRM ORDERS}}
We first verify Theorem \ref{THRM ORDERS} for the goal oriented method. Figure \ref{FIG TOY THRM VERIFY} shows results for the Crank-Nicolson scheme with Implicit Euler for error estimation, trapezoidal rule for quadrature and a range of different density functions. With $p, \,\,\hat{p} = (2,\,\,1)$ and $r = 2$, we expect at least $e^J = \mathcal{O}(\tau)$, which the plots clearly show.
\begin{figure}[ht!]
\centering
\includegraphics[scale = 0.19]{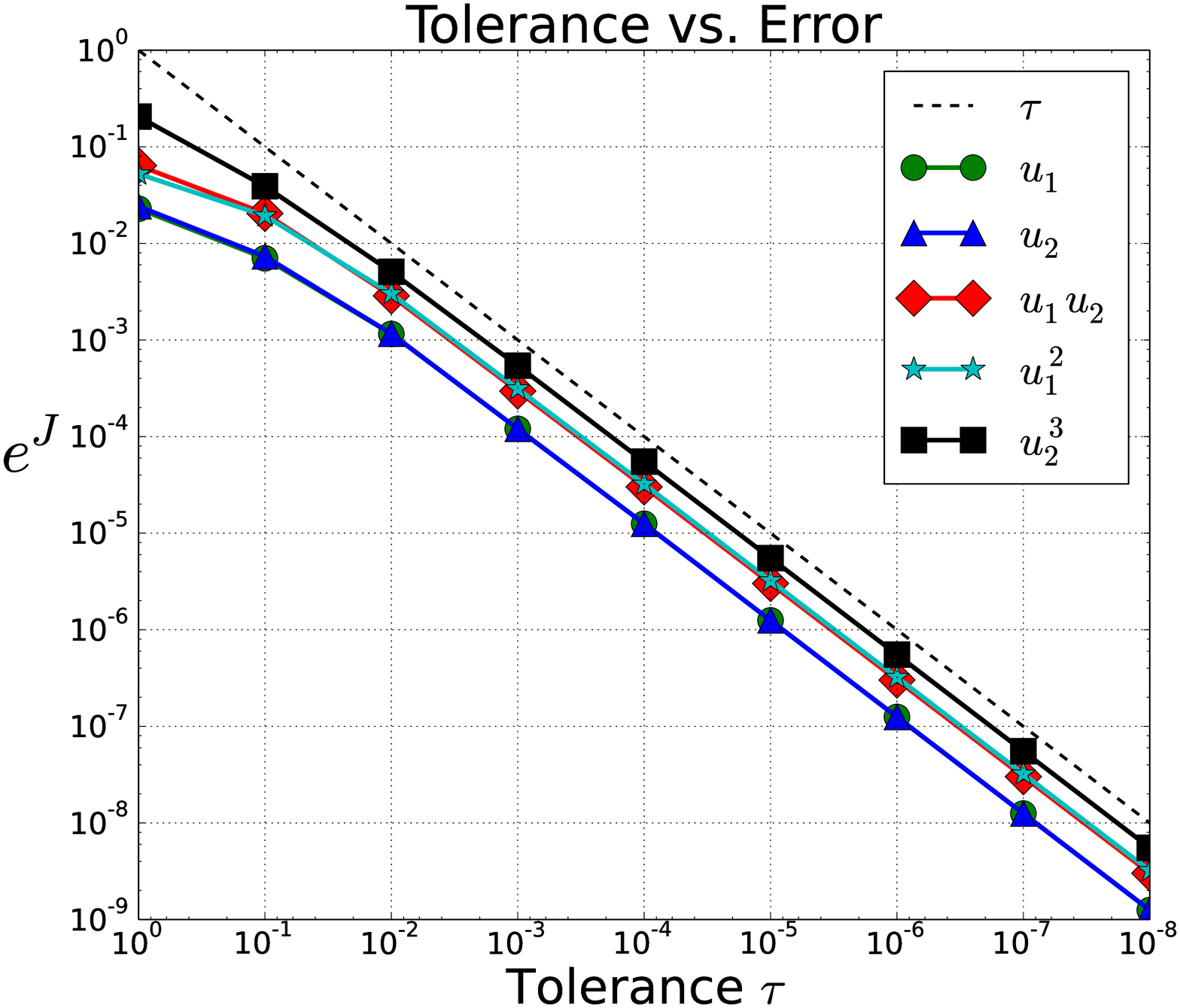}
\includegraphics[scale = 0.19]{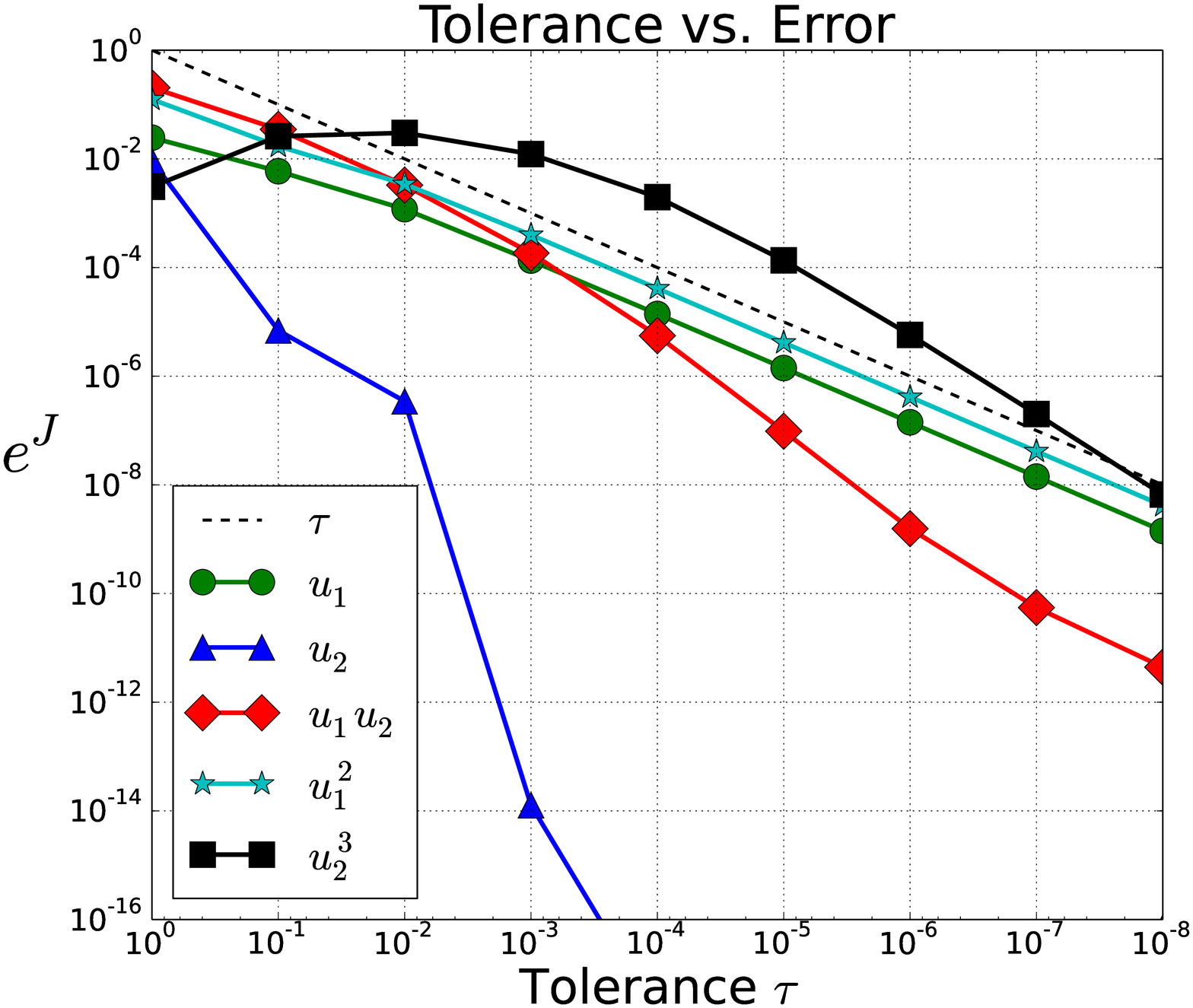}
\caption{Verification of Theorem \ref{THRM ORDERS} using the goal oriented adaptive method on problem \eqref{EQ TOY MAIN} for $k = -1$ (left) and $k = -100$ (right), the legend shows $j(t, \bm{u})$.}
\label{FIG TOY THRM VERIFY}
\end{figure}

Further we consider fourth order schemes $p = r = 4$ for the goal oriented adaptive method with time-dependent density functions $j$. We use Simpson's rule for quadrature and the classical Runge-Kutta scheme for time-integration. As embedded scheme we use the weights $\hat{b} = \frac{1}{3}(1, 1, 0, 1)^T$, which give a second order (third order for autonomous systems) solution. To get a fourth order solution in $t_{n} + \Delta t_n/2$ needed by the Simpson rule we use the weights $\bm{b}^* = \frac{1}{24}(5, 4, 4, -1)^T$. As the test problem \eqref{EQ TOY MAIN} is autonomous we have $p, \,\,\hat{p} = (4, 3)$. With $r = 4$ and 4th order solutions for all evaluation points of the quadrature scheme, we expect to get $e^J = \mathcal{O}(\tau) = \mathcal{O}(N^{-4})$ from Theorem \ref{THRM ORDERS}. This can be observed in Figure \ref{FIG TOY ORDER 4}.
\begin{figure}[ht!]
\centering
\includegraphics[scale = 0.19]{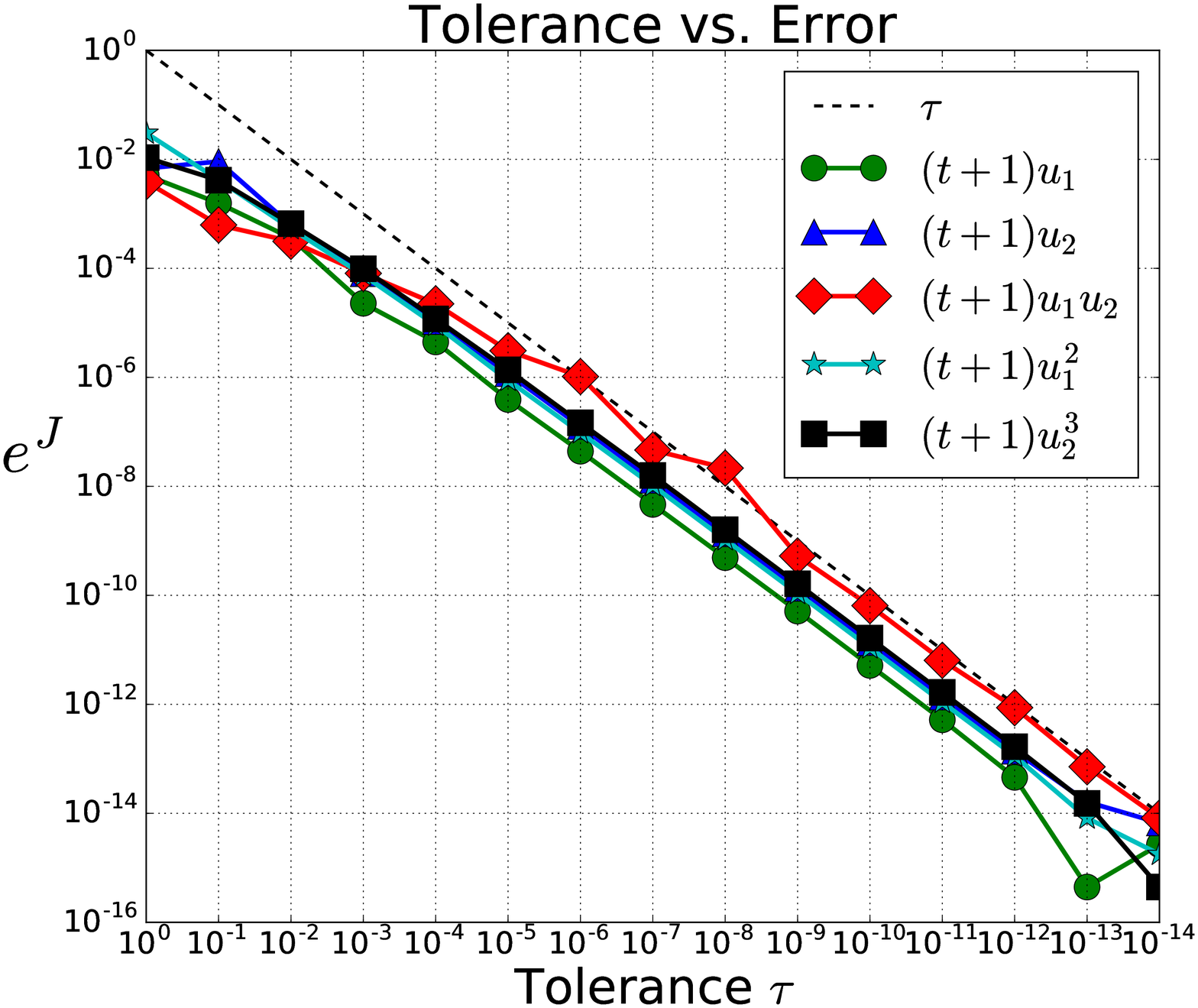}
\includegraphics[scale = 0.19]{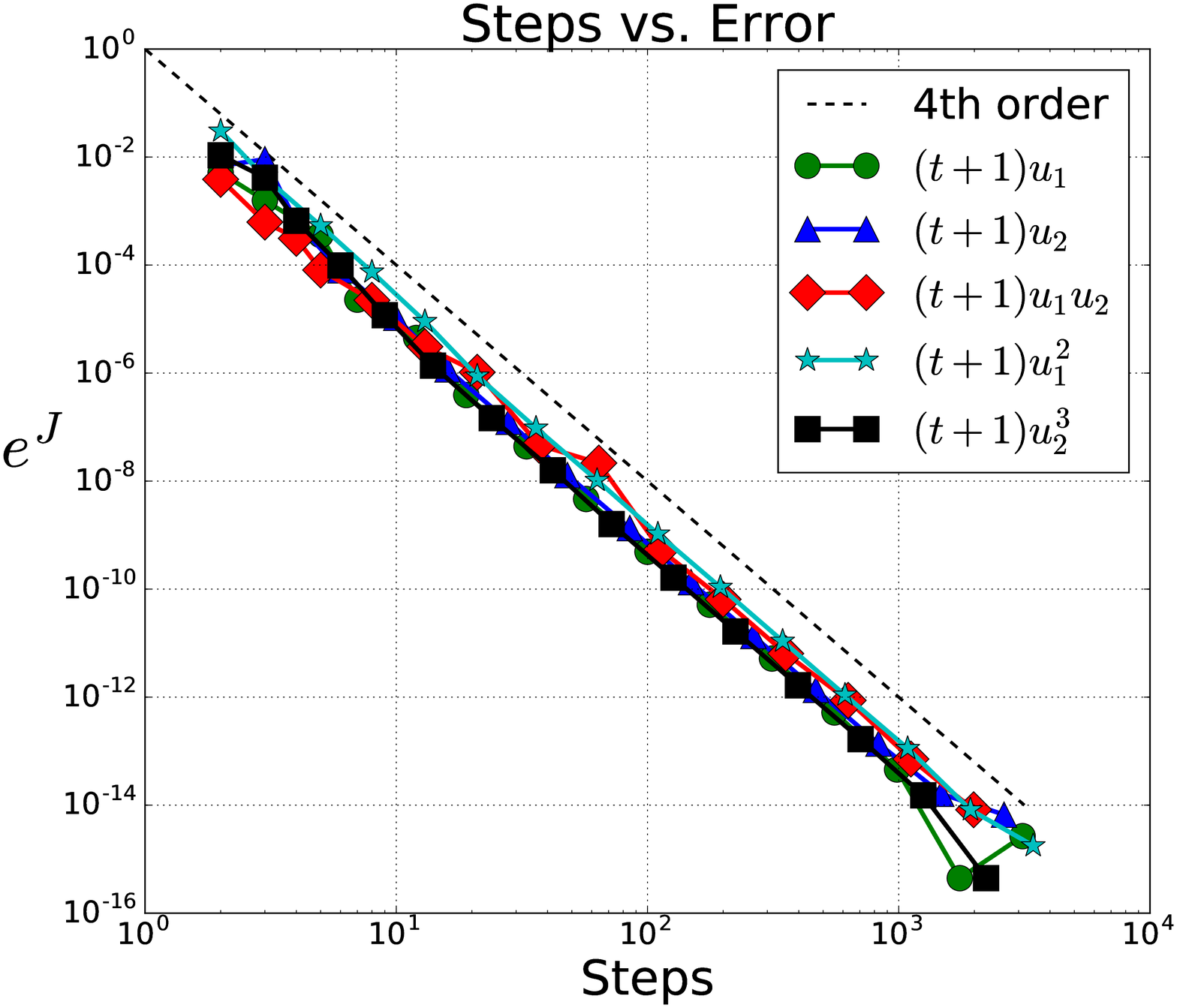}
\caption{Verification of Theorem \ref{THRM ORDERS} using goal oriented adaptive method on problem \eqref{EQ TOY MAIN} for $k = -1$, using 4-th order schemes, the legend shows $j(t, \bm{u})$.}
\label{FIG TOY ORDER 4}
\end{figure}
%
\subsubsection{Method comparison and performance tests}\label{SEC TOY PERF CMP}
We now compare the DWR method with the local error based classic and goal oriented method. For the DWR method we additionally have the estimate $\eta(u_h)$ of the error $e^J$, which we denote by "DWR Est" in Figures, the actual error is denoted by "DWR Err". We only consider the final grid with DWR Est $= \eta(u_h) \leq \tau$. We use second order time-integration for both DWR and the local error methods. As DWR requires a variational formulation, we use the Crank-Nicolson scheme for time-integration and for the local error based methods Implicit Euler for error estimation. For quadrature we use the trapezoidal rule. 

We compare methods in terms of computational efficiency (error vs. computational time spent) and grid quality (error vs. number of timesteps). We consider the density functions $j(t, \bm{u}) = u_1$ for $k \in \{ -1, -100\}$ (Figures \ref{FIG TOY U1 K 1}, \ref{FIG TOY U1 K 100}) and $j(t, \bm{u})= u_2$ for $k = -1$ (Figure \ref{FIG TOY U2 K 1}). Results for DWR are considered first and the local error based adaptive methods are then discussed based on the results of section \ref{SEC ERR PROP}.
\begin{figure}[ht!]
\centering
\includegraphics[scale = 0.19]{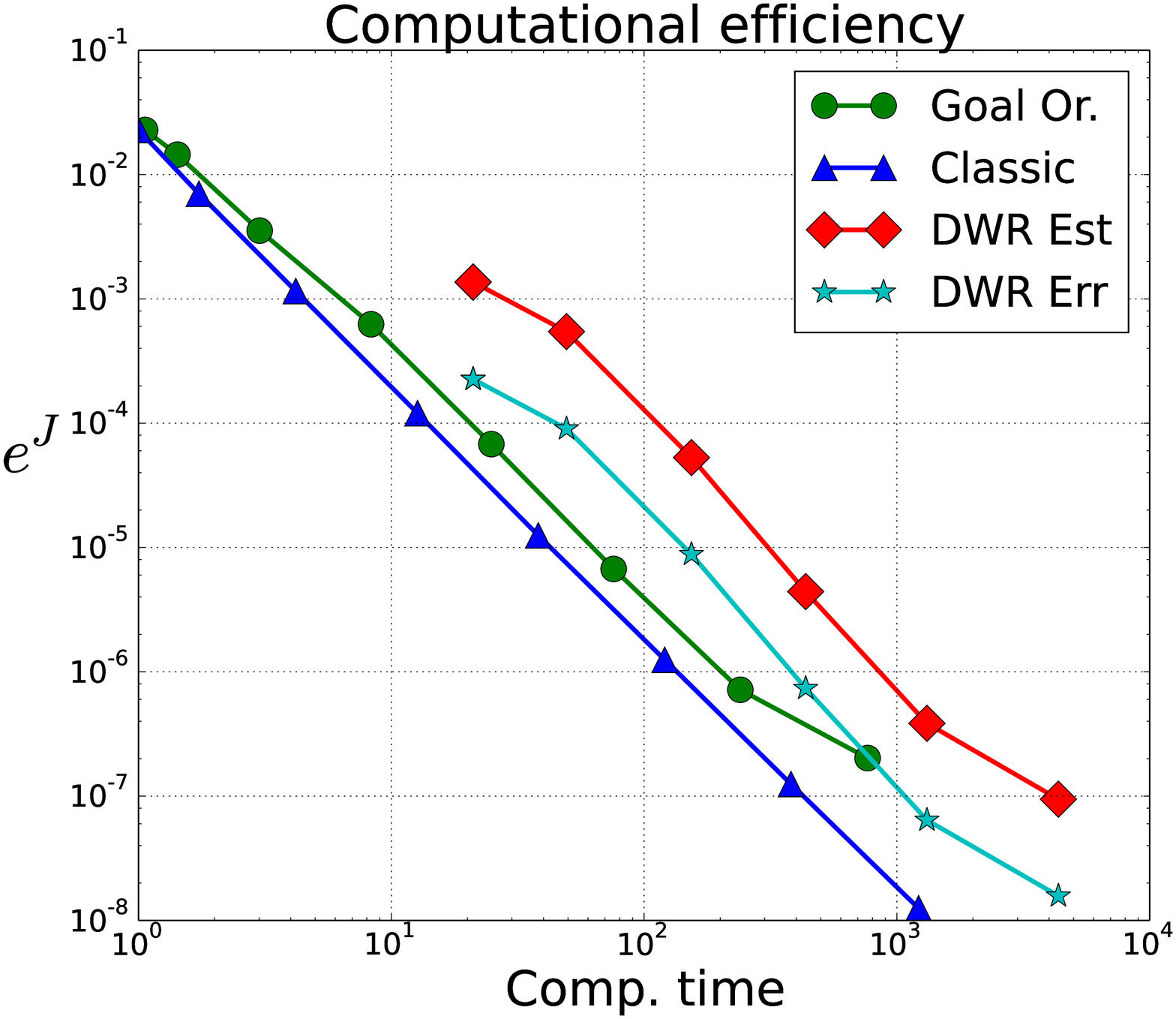}
\includegraphics[scale = 0.19]{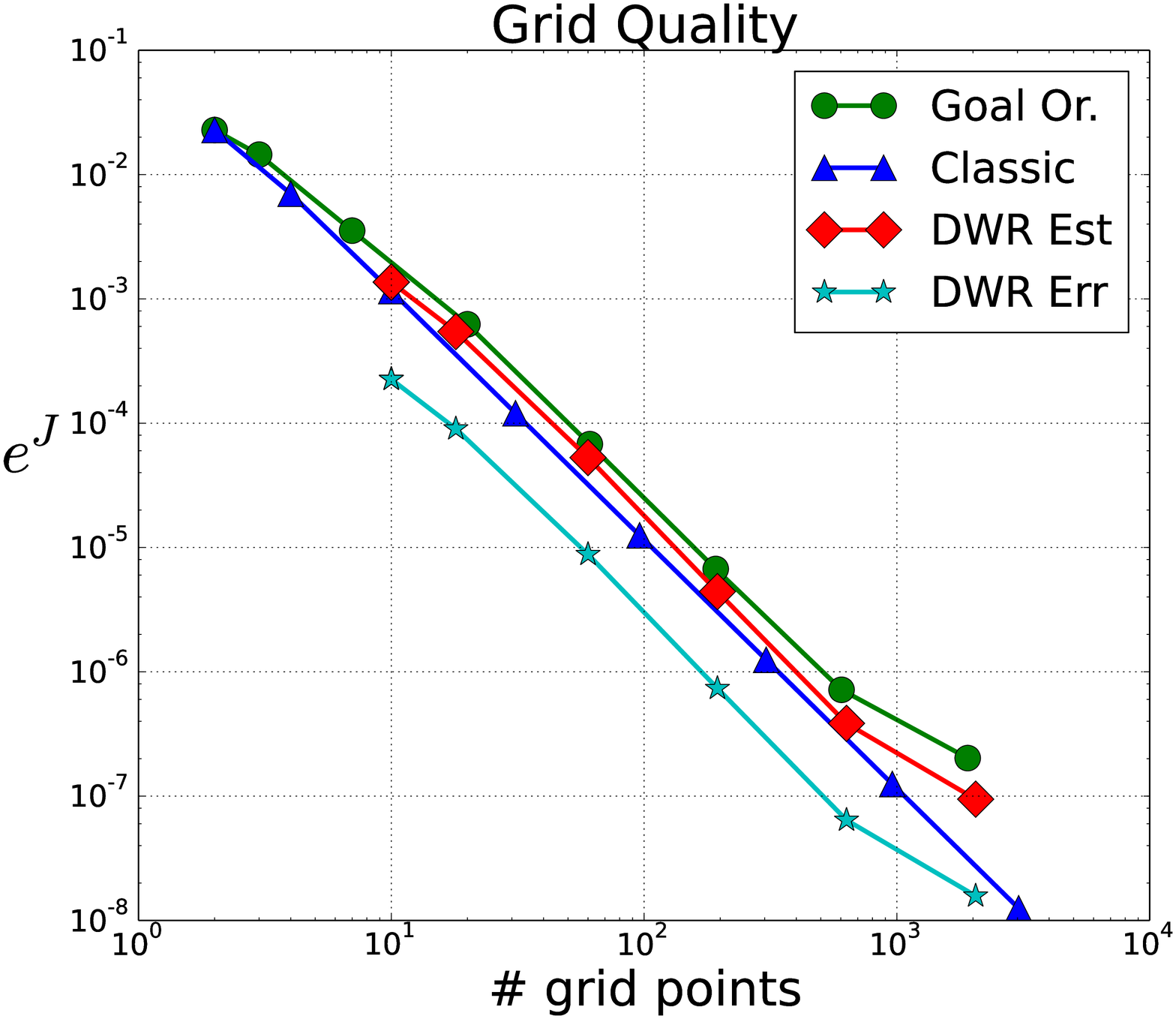}
\caption{Performance comparison of the various methods for problem \eqref{EQ TOY MAIN} for $k =-1$ and $j(t, \bm{u}) = u_1$.}
\label{FIG TOY U1 K 1}
\end{figure}
\begin{figure}[ht!]
\centering
\includegraphics[scale = 0.19]{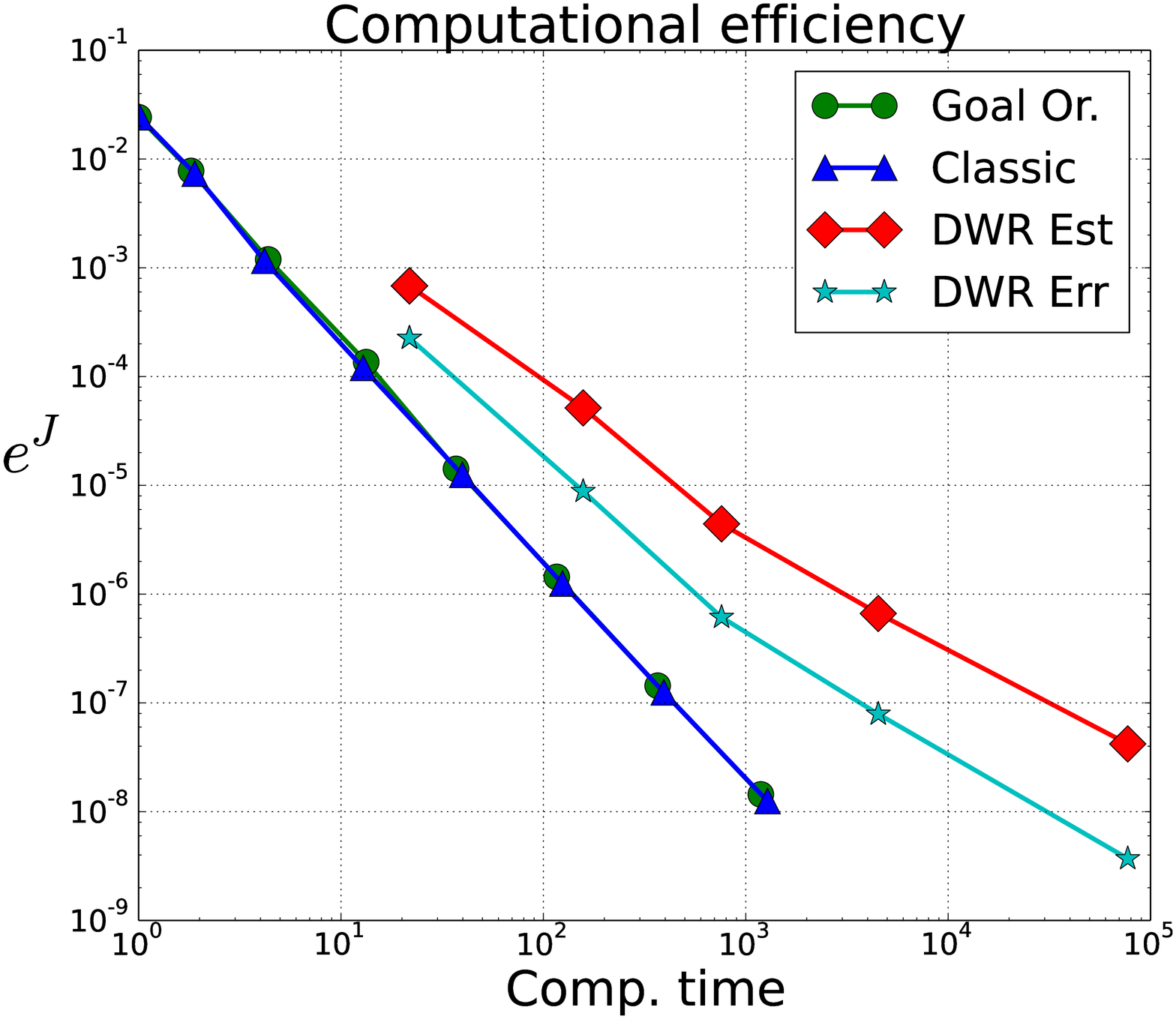}
\includegraphics[scale = 0.19]{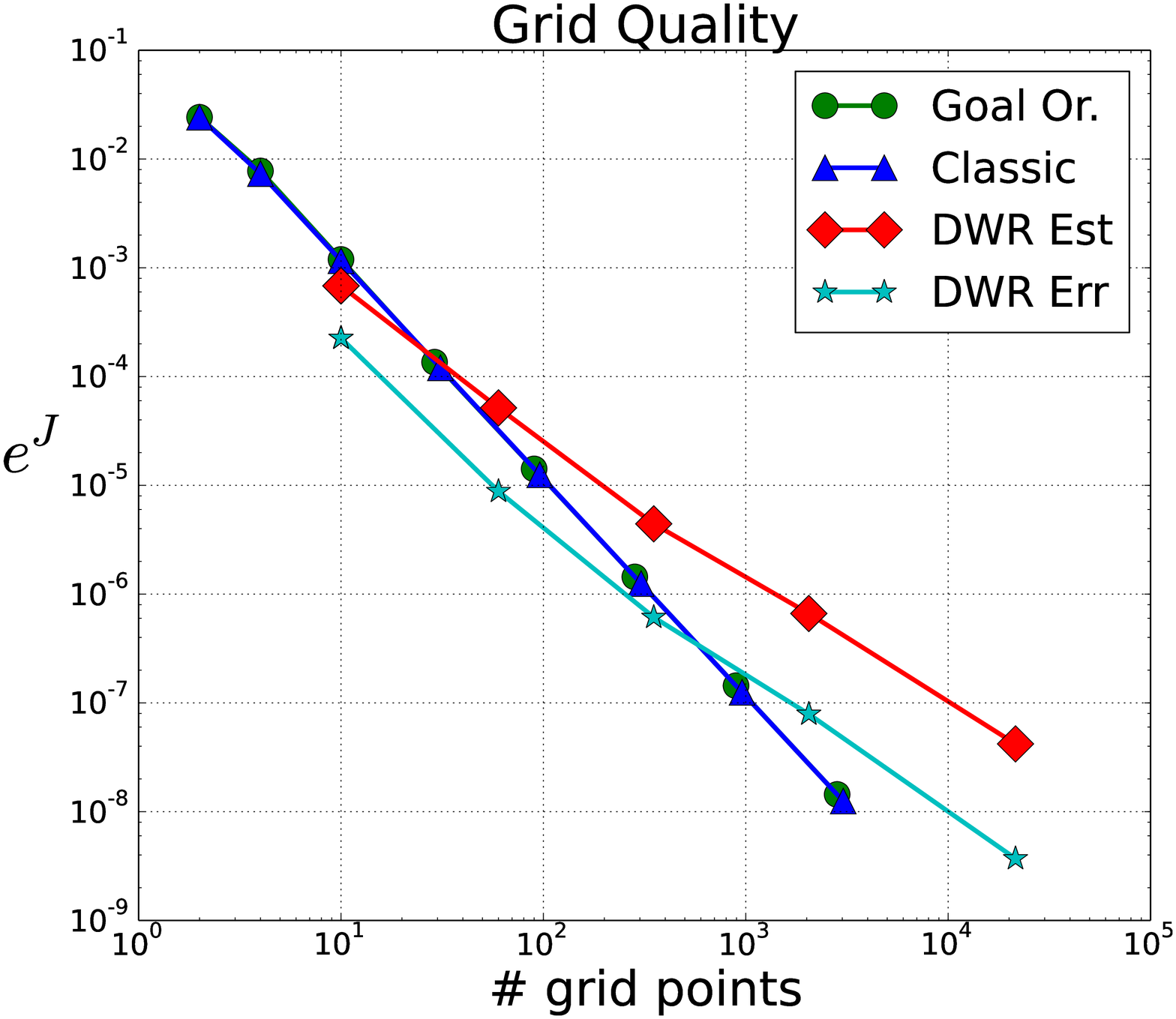}
\caption{Performance comparison of the various methods for problem \eqref{EQ TOY MAIN} for $k =-1$ and $j(t, \bm{u}) = u_2$.}
\label{FIG TOY U2 K 1}
\end{figure}
\begin{figure}[ht!]
\centering
\includegraphics[scale = 0.19]{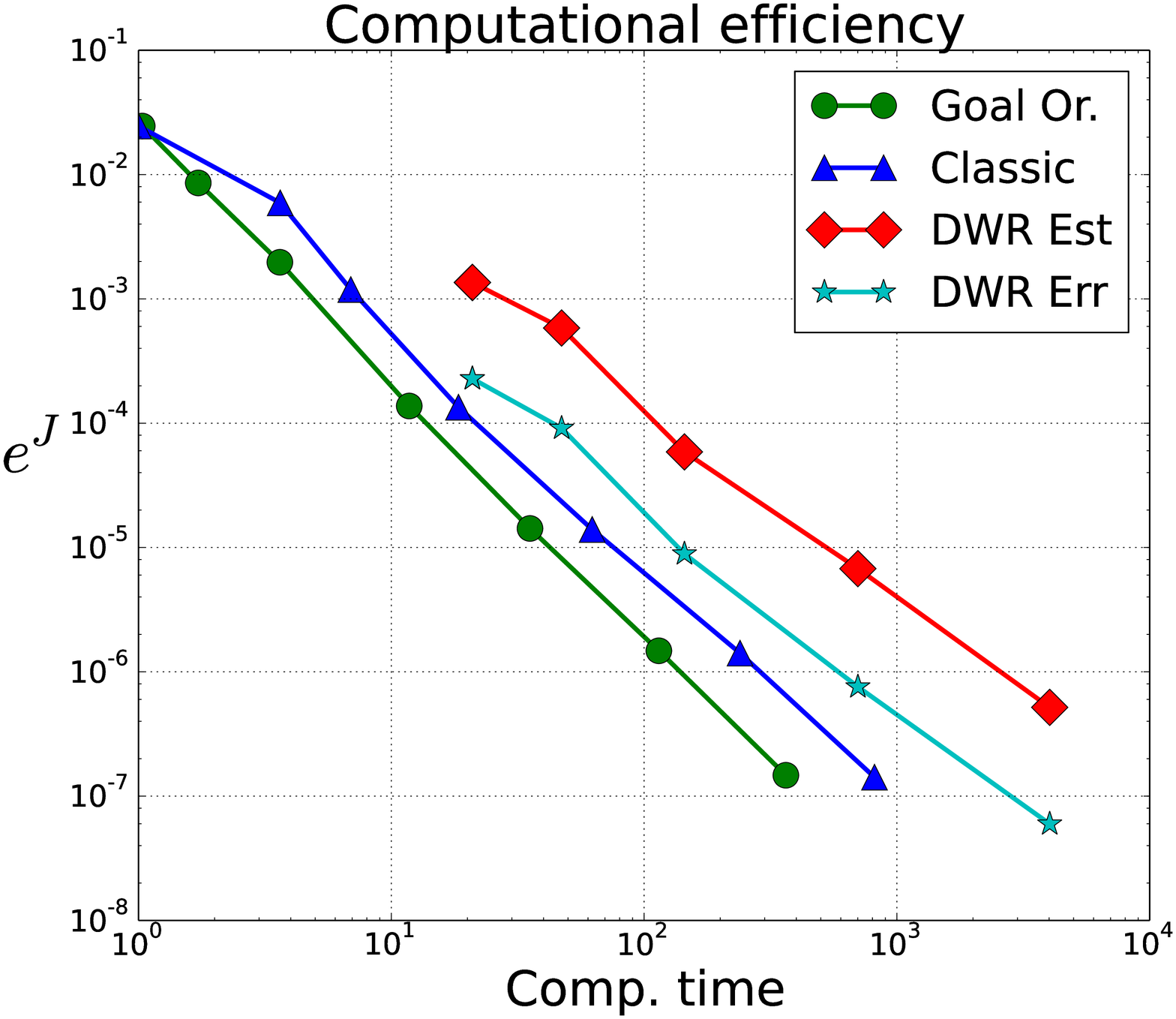}
\includegraphics[scale = 0.19]{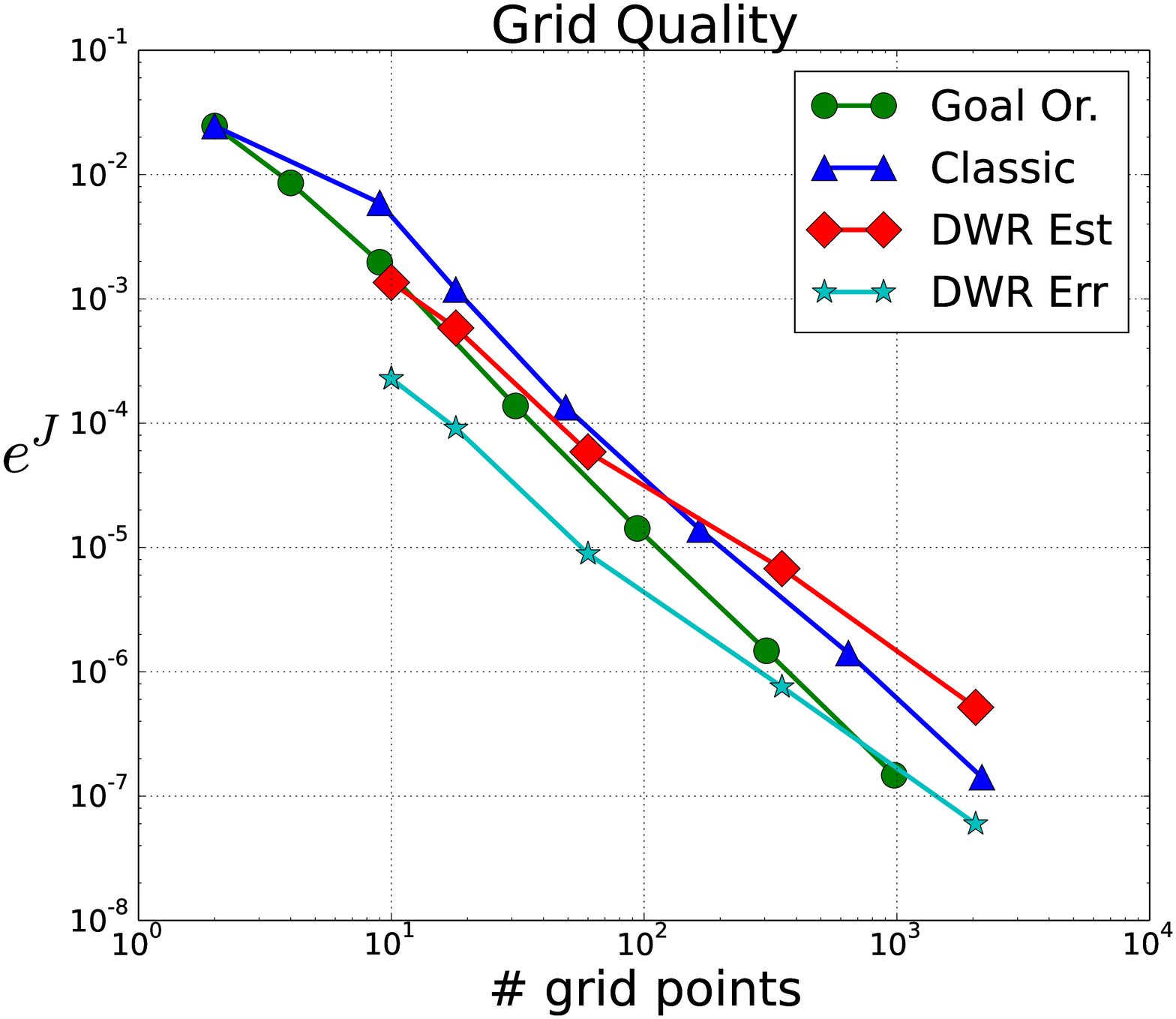}
\caption{Performance comparison of the various methods for problem \eqref{EQ TOY MAIN} for $k =-100$ and $j(t, \bm{u}) = u_1$.}
\label{FIG TOY U1 K 100}
\end{figure}

Looking at Figures \ref{FIG TOY U1 K 1} - \ref{FIG TOY U1 K 100} and considering the actual error (DWR Err), we see the method produces the best grids. This is expected, since the method uses global grid adaptation. But the DWR method is significantly slower in performance due to the need of solving adjoint equations in computing the error estimate.

The differences in the local error methods have to be discussed for each case individually. For $k=-1$ and $j(t, \bm{u}) = u_1$, see Figure \ref{FIG TOY U1 K 1}, the derivative of $u_1$ is slightly smaller, due to the additional off-diagonal term. Thus only controlling the error in the first component under-resolves the second component, which is relevant to $J(\bm{u})$ due to coupling. While not immediately evident, we do not fulfill the criterion \eqref{EQ J PHIMIN} needed for convergence in the QoI, we have 
\begin{equation}\label{EQ TOY PRINC ERR FCT}
j(t, \bm{\phi}(t, \bm{u}(t))) = \frac{1}{2} e^{-t} (t - 1),
\end{equation}
meaning the error estimate vanishes at $t = 1$ for $\tau \rightarrow 0$. As a result we do not have convergence in the QoI for $\tau \rightarrow 0$, since the timestep taken at $t = 1$ will tend to infinity. This trend can be observed when looking at the timesteps over time in Figure \ref{FIG TOY U1 K 1 STEPS}, which form an upward cusp. We are, however, using an extremely small tolerance of $\tau = 10^{-14}$ and have the error $e^J \approx 4 \cdot 10^{-13}$, which is already close to machine zero. This shows that the requirement \eqref{EQ J PHIMIN} may not be a strict requirement on convergence in the QoI in practice for some problems.
\begin{figure}[ht!]
\centering
\includegraphics[scale = 0.19]{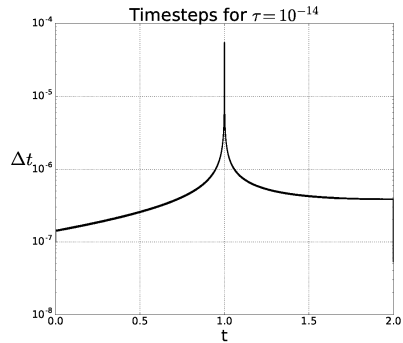}
\caption{Timesteps for numerical solution of \eqref{EQ TOY MAIN} for $j(t, \bm{u}) = u_1$, $k = -1$ and $\tau = 10^{-14}$. A cusp at $t \approx 1$ can be observed, where the principal error function \eqref{EQ TOY PRINC ERR FCT} has a zero. Here we did not use the step-size limiter \eqref{EQ DT LIMITER}.}
\label{FIG TOY U1 K 1 STEPS}
\end{figure}

In the case of $k=-1$ and $j(t, \bm{u}) = u_2$, see Figure \ref{FIG TOY U2 K 1}, we do control the error in the fastest process with the goal oriented method. Thus the chosen timesteps sufficiently resolve all processes. The results show that the two local error based methods have grids of identical quality and require the same computational effort.

For $k = -100$ and $j(t, \bm{u}) = u_1$, see Figure \ref{FIG TOY U1 K 100}, we similarly to the case of $k = -1$ do not control the fastest process, but the impact of $u_2$ on $J(\bm{u})$ is small. It turns out the efficiency gain in not properly resolving the second component is worth the additional error, resulting in a more efficient method by a factor of around two.
%
\subsection{Convection-diffusion equation}
%
Moving to a problem involving a spatial component, we look at a linear convection-diffusion equation
\begin{align}
\partial_t u(t, \bm{x}) + a \, \bm{v} \cdot \nabla u(t, \bm{x}) - \gamma \Delta u(t, \bm{x}) = f(t,\bm{x}), \quad & (t, \bm{x}) \in [t_0, t_e] \times \Omega, \nonumber\\
u(t_0, \bm{x}) = u_0(\bm{x}), \quad & \bm{x} \in \Omega, \label{EQ HEAT PROP}\\
\nabla u \cdot \bm{n} = - c \, \bm{u}(t, \bm{x}), \quad & (t, \bm{x}) \in [t_0, t_e] \times \partial \Omega \nonumber.
\end{align}
We want to model the case of having error build-up in the nullspace of $j(t, u)$, which is transported into its image. We consider the domain $\Omega = [0,3] \times [0,1]$ and restrict the source term $f$ to $\Omega_f = [0.25, 0.75] \times [0.25, 0.75]$. As QoI we consider 
\begin{equation}\label{EQ HEAT J}
J(u) = \int_{t_0}^{t_e} j(t, u(t)) dt, \quad j(t, u(t)) = \int_{\Omega_*} \frac{u(t, \bm{x})}{t_e - t_0} dx.
\end{equation}
with $\Omega_* = [2.25, 2.75] \times [0.25, 0.75]$. For a visualization of the spatial domain, see Figure \ref{FIG HEAT DOMAINS}, for the time-domain we use $[t_0, t_e] = [0, 6]$.

\begin{figure}[ht!]
\begin{center}
\begin{tikzpicture}[scale = 2.5]
\draw [-] (1,1) -- (4,1) -- (4,2) -- (1,2) -- (1,1);
\draw [-] (1.25,1.25) -- (1.75,1.25) -- (1.75,1.75) -- (1.25,1.75) -- (1.25,1.25);
\draw [-] (3.25,1.25) -- (3.75,1.25) -- (3.75,1.75) -- (3.25,1.75) -- (3.25,1.25);

\node at (1.5, 1.5) {$\Omega_f$};
\node at (3.5, 1.5) {$\Omega_*$};
\node at (2.5, 1.1) {$\Omega$};

\draw [->] (2.3,1.5) -- (2.7, 1.5); \node at (2.5, 1.55) {$\bm{v}$};

\draw [-] (1,1) -- (0.95, 1); \node at (0.85, 1) {0};
\draw [-] (1,1.25) -- (0.95, 1.25); \node at (0.75, 1.25) {0.25};
\draw [-] (1,1.75) -- (0.95, 1.75); \node at (0.75, 1.75) {0.75};
\draw [-] (1,2) -- (0.95, 2); \node at (0.85, 2) {1};

\draw [-] (1,1) -- (1, 0.95); \node at (1, 0.85) {0};
\draw [-] (1.25,1) -- (1.25, 0.95); \node at (1.25, 0.85) {0.25};
\draw [-] (1.75,1) -- (1.75, 0.95); \node at (1.75, 0.85) {0.75};
\draw [-] (3.25,1) -- (3.25, 0.95); \node at (3.25, 0.85) {2.25};
\draw [-] (3.75,1) -- (3.75, 0.95); \node at (3.75, 0.85) {2.75};
\draw [-] (4,1) -- (4, 0.95); \node at (4, 0.85) {3};

\end{tikzpicture}
\end{center}
\caption{Geometry of $\Omega$ for the convection-diffusion equation problem \eqref{EQ HEAT PROP}.}
\label{FIG HEAT DOMAINS}
\end{figure}
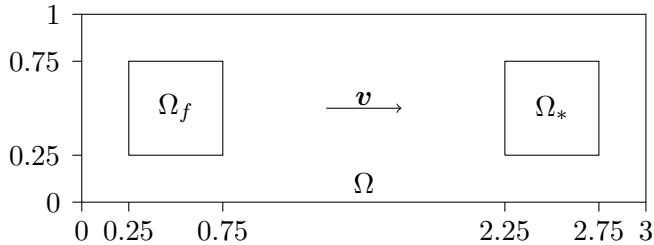
We use the source term
\begin{equation*}
f(t, \bm{x}) = 
\begin{cases}
5 \, t^3, & \bm{x} \in \Omega_f \,\, \text{and}\,\, t < 1, \\
5\, (2 - t)^3, & \bm{x} \in \Omega_f \,\, \text{and}\,\, 1 \leq t < 2,\\
0, & \bm{x} \notin \Omega_f \,\,\text{or}\,\,\,\,\,\, 2 \leq t,
\end{cases}
\end{equation*}
providing a spike-shaped build-up in the first $2$ time units. The remaining parameters are $a = 0.5$, $\gamma = 0.01$, $c = 0.15$ and $\bm{v} = (1, 0)^T$. We use the initial condition $u_0(\bm{x}) = 1$. Since we do not have an analytical solution, we use as reference solution from using the classic adaptive method with $\tau_{\text{ref}} = \tau_{\text{min}}/10$, where $\tau_{\text{min}}$ is the minimal tolerance for which tests are done.
%
%
\subsubsection*{Discretization}
%
For our convection-diffusion problem we have a weak solution in the space
\begin{equation*}
V := H^1(t_0,t_e: L^2(\Omega)) \cap L^2(t_0, t_e: H^1(\Omega)),
\end{equation*}
see \cite{renardy2006introduction}. We discretize time along the points $t_n$ with $I_n := [t_n, t_{n+1}]$ and space by $2\,(3\cdot32)\cdot32 = 6144$ regular triangular cells $K$ defining the finite element mesh. We define the global finite element space by
\begin{align*}
V_{h,k} = \{ v \in L^\infty(t_0, t_e; H^1(\Omega)): \,&
v(\cdot,t)|_{Q_K^n} \in Q^1(K), \\
& v(x, \cdot)|_{Q_K^n} \in P^q(I_n), \forall \,\,Q_K^n\},
\end{align*}
where $P^q(I_n)$ is the space of polynomials on $I_n$ of degree up to $q$ and $Q^1(K)$ being the space of polynomials on $K$ with partial degrees up to $1$. In this space the variational formulation becomes
\begin{equation*}
A_h(u_h, \phi_h) = F(\phi_h)
\end{equation*}
for all $\phi \in V_{h,k}$ with the bilinear form
\begin{equation*}
A_h(u_h, \phi_h) := \int_{t_0}^{t_e}
(\partial_t u_h, \phi_h) +
a (\bm{v} \cdot \nabla u_h, \phi_h) - 
\gamma (\Delta u_h, \phi_h) dt
\end{equation*}
and right-hand side
\begin{equation*}
F(\phi_h) = \int_{t_0}^{t_e} (f(t), \phi_h)dt.
\end{equation*}
The weak formulation is
\begin{equation*}
\int_{t_0}^{t_e} (\partial_t u_h + a \bm{v} \cdot \nabla u_h - f, \phi_h) + 
\gamma (\nabla u_h, \nabla \phi_h) + 
\gamma \, c (u_h, \phi_h)_{\partial \Omega} dt = 0,
\end{equation*}
from which one can directly write down the $\theta$-method yielding both Crank-Nicolson and Implicit Euler.

We have the adjoint equation
\begin{align*}
- z_t(t, \bm{x}) - a \,\bm{v} \cdot \nabla \bm{z}(t, \bm{x}) - \gamma \Delta z(t, \bm{x}) = \textstyle{\frac{1}{t_e - t_0}}\Big|_{\Omega_*}, & \quad (t,\bm{x}) \in [t_0, t_e] \times \Omega,\\
z(t_e, \bm{x}) = 0, & \quad \bm{x} \in \Omega,\\
\nabla z(t, \bm{x}) \cdot \bm{n} = - \frac{a}{\gamma} \, z(t, \bm{x})\, \bm{v} \cdot \bm{n} - c \,z(t, \bm{x}), & \quad (t,\bm{x}) \in [t_0, t_e] \times \partial \Omega.
\end{align*}
The weak formulation is
\begin{align*}
\int_{t_0}^{t_e} & (z_h + a \,\bm{v}\cdot\nabla z_h, \phi_h)_{\Omega} + \gamma (\nabla z_h, \nabla \phi_h)_{\Omega} \,+ \\
& (\textstyle{\frac{1}{t_e - t_0}}, \phi_h)_{\Omega_*} - (c\,\gamma\, z_h + a\,z_h\,\bm{v}\cdot \bm{n}, \phi_h)_{\partial \Omega} dt = 0.
\end{align*}
%
%
\subsubsection*{DWR Estimate}
%
We have 
\begin{align*}
e^J & = \left|\int_{t_0}^{t_e} \int_{\Omega} A_h(u_h, z - z_h) - F(z_h) dx\,dt\right|\\
& = \left|\int_{t_0}^{t_e} \int_{\Omega} ((u_h)_t + a\,\bm{v}\cdot \nabla u_h - \gamma \Delta u_h - f)(z - z_h) dx\, dt\right|,
\end{align*}
where we approximate $z\approx z_h^+$ using a finer grid in time. Splitting this by timesteps gives
\begin{equation*}
e^J \lesssim \sum_{n=0}^{N-1} \Big|\int_{t_n}^{t_{n+1}} \underbrace{\int_{\Omega} ((u_h)_t + a\,\bm{v}\cdot \nabla u_h - \gamma \Delta u_h - f)(z_h^+ - z_h) dx}_{=: R_z(t)}\, dt\Big|.
\end{equation*}
We use the composite trapezoidal rule to get the error estimate
\begin{equation*}
e^J \lesssim \eta_h(u_h) := \sum_{n=0}^{N-1} \frac{\Delta t_n}{4}\left| R_z(t_n) + 2 R_z(t_n + \Delta t_n/2) + R_z(t_{n+1})\right|,
\end{equation*}
using linear interpolation for $u_h$ and $z_h$ in computing $R_z(t_n + \Delta t_n/2)$.
%
\subsubsection{Method comparison and performance tests}
%
We use the same schemes for time-integration as in section \ref{SEC TOY PERF CMP}. We again compare DWR with the two local error based adaptive methods. The way we set up the problem, we expect the goal oriented method to perform poorly. Due to the source term being in the nullspace of $j(t, u)$, the resulting timesteps will not sufficiently resolve it. The convection transports the build-up from the source term and its error into the image of $j(t, u)$. This leads to an increase in error, which can no longer be controlled by the step-size. 
\begin{figure}[ht!]
\centering
\includegraphics[scale = 0.19]{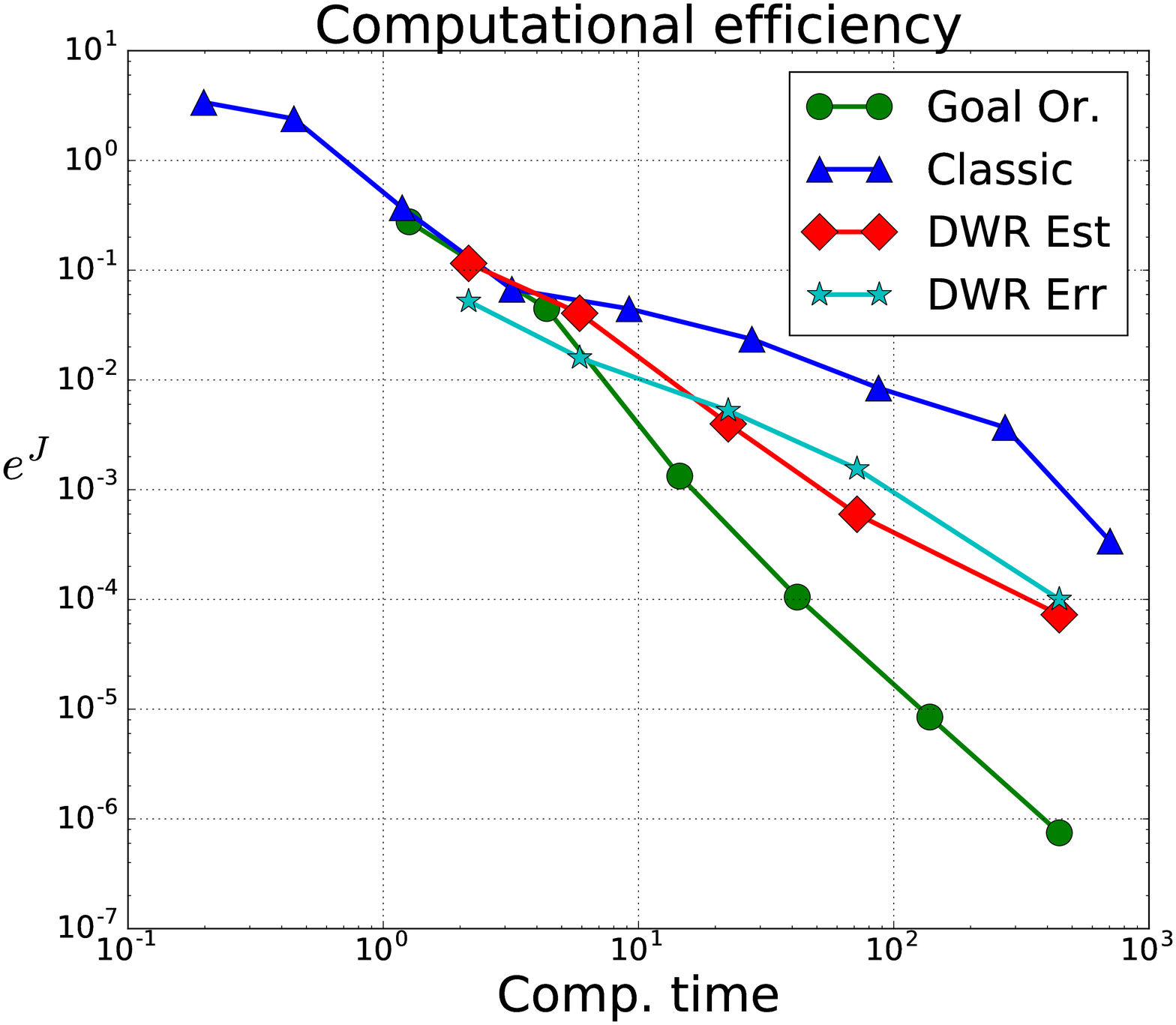}
\includegraphics[scale = 0.19]{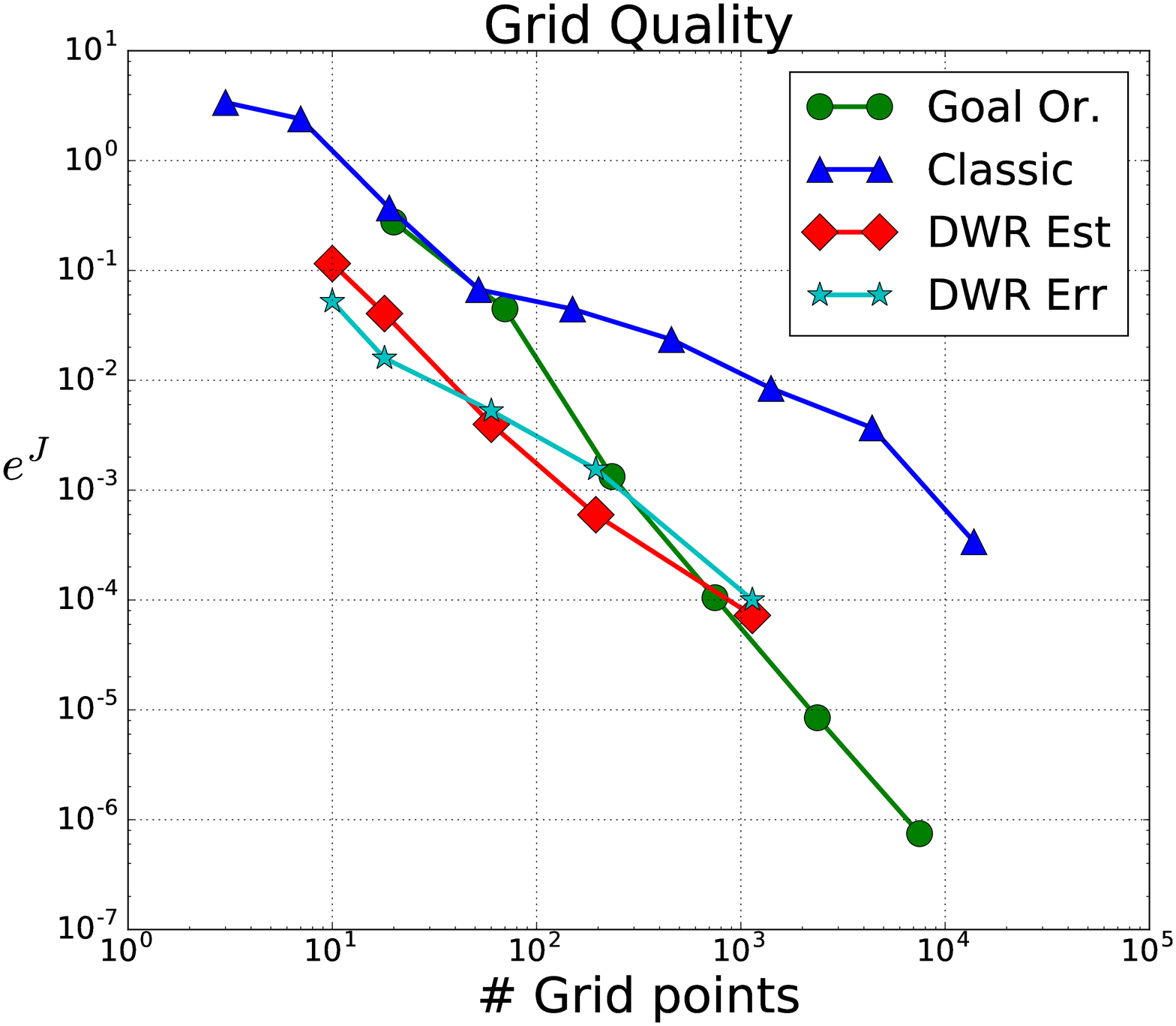}
\caption{Performance comparison of the various methods for problem \eqref{EQ HEAT PROP} with QoI \eqref{EQ HEAT J}.}
\label{FIG HEAT FWD}
\end{figure}

The results can be seen in Figure \ref{FIG HEAT FWD}. One can observe the classic adaptive method performs fine and the goal oriented adaptive method shows the expected poor performance. In Figure \ref{FIG HEAT TIMESTEPS} one can see the timesteps chosen by the goal oriented method are too large to resolve the source term. Nevertheless we have convergence in the QoI with $e^J = \mathcal{O}(\tau)$, as predicted by Theorem \ref{THRM ORDERS}, see Figure \ref{FIG HEAT TIMESTEPS}.
\begin{figure}[ht!]
\centering
\includegraphics[scale = 0.19]{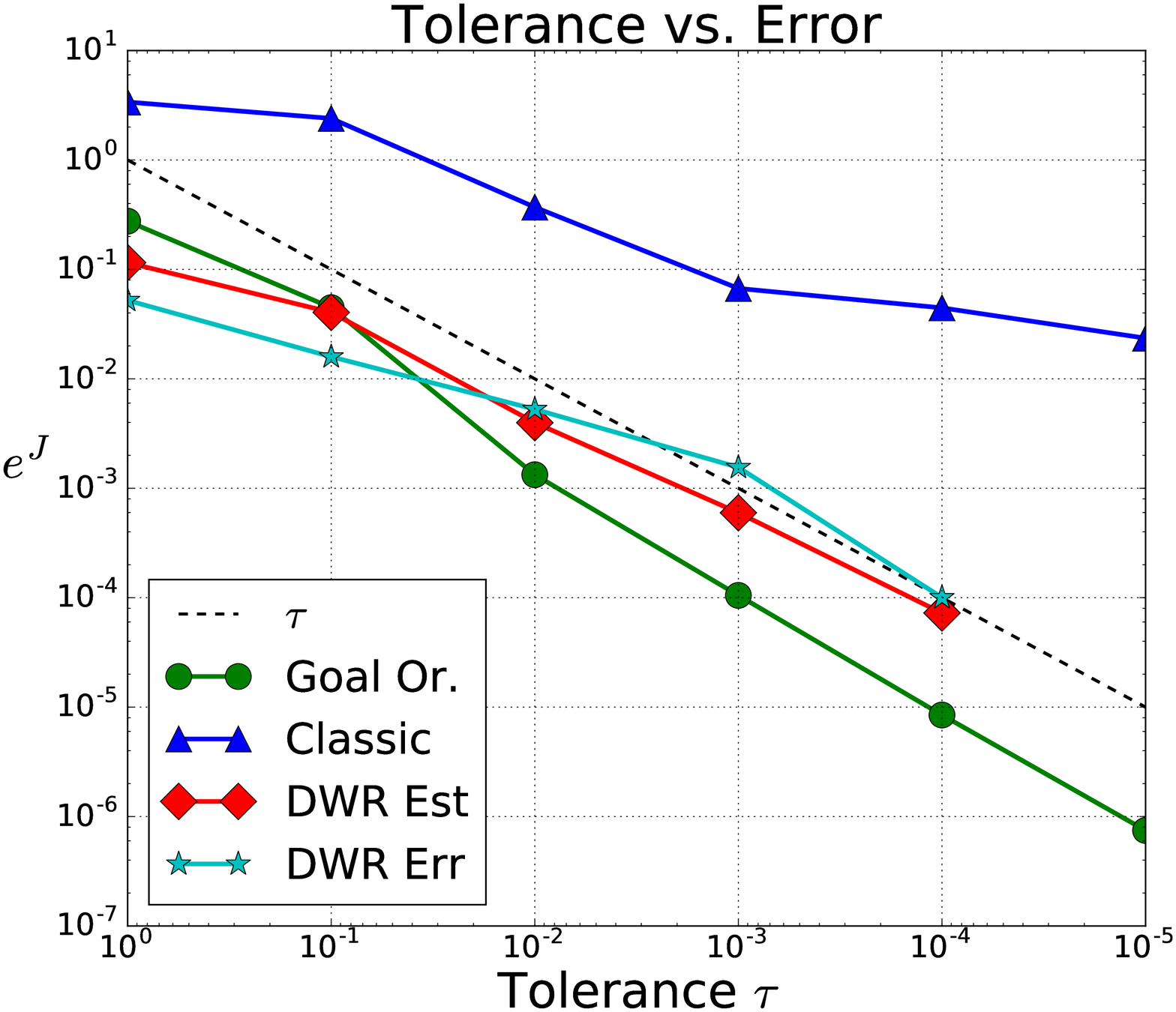}
\includegraphics[scale = 0.19]{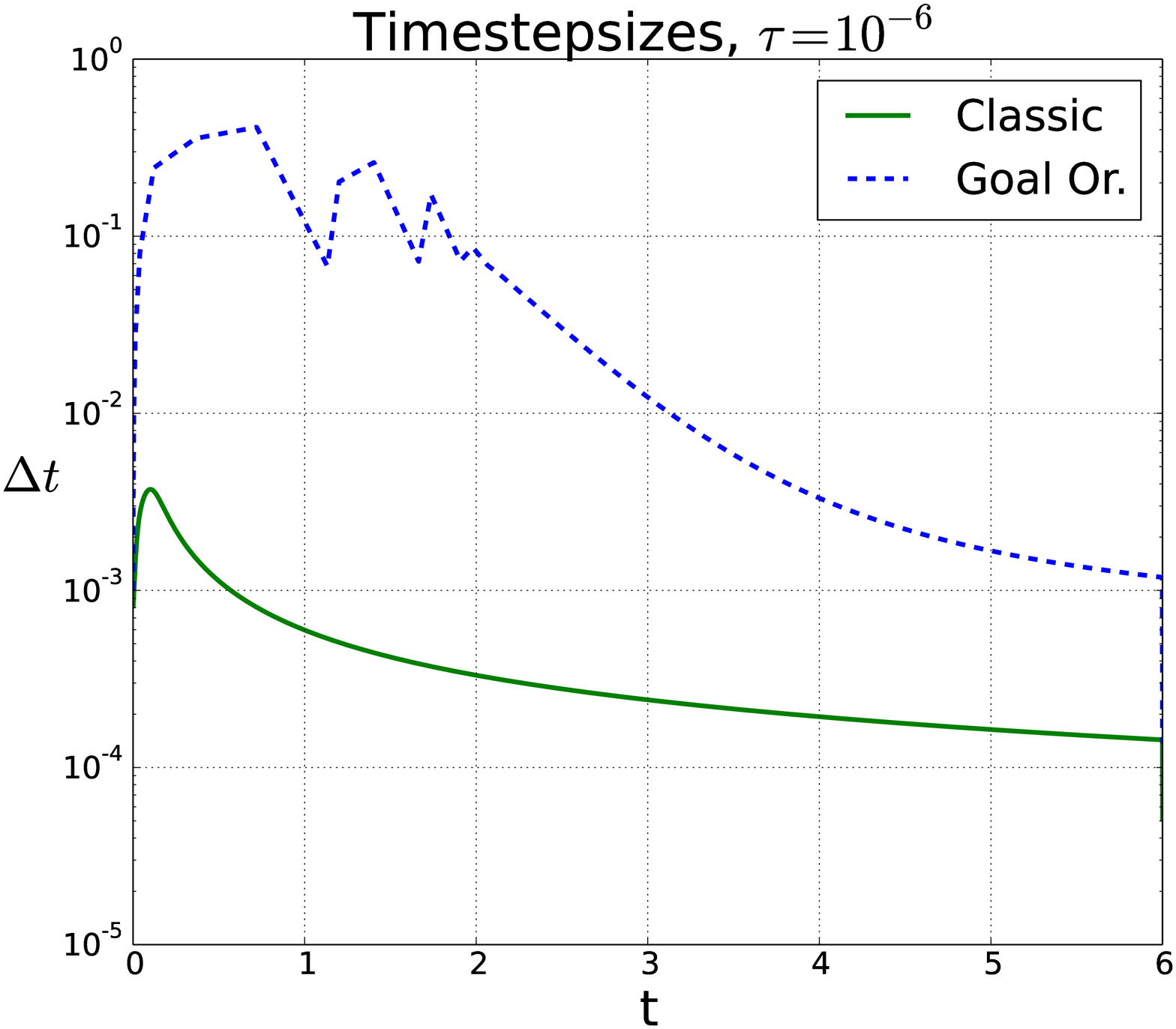}
\caption{Solving problem \eqref{EQ HEAT PROP} with QoI \eqref{EQ HEAT J}. Left: Tolerance over error for the goal oriented adaptive method; Right: Timesteps chosen by the different methods for $\tau = 10^{-6}$.}
\label{FIG HEAT TIMESTEPS}
\end{figure}

The DWR method is computationally expensive, but gives high quality grids. Here, we used it to only adapt the grid in time to get a fair comparison with the other methods.

Changing the sign of the convection term we expect good results for the goal oriented method, since it is no longer required to properly resolve the source term. Considering only $[t_0, t_e] = [0, 3]$ we get the results seen in Figure \ref{FIG HEAT BWD}.
\begin{figure}[ht!]
\centering
\includegraphics[scale = 0.19]{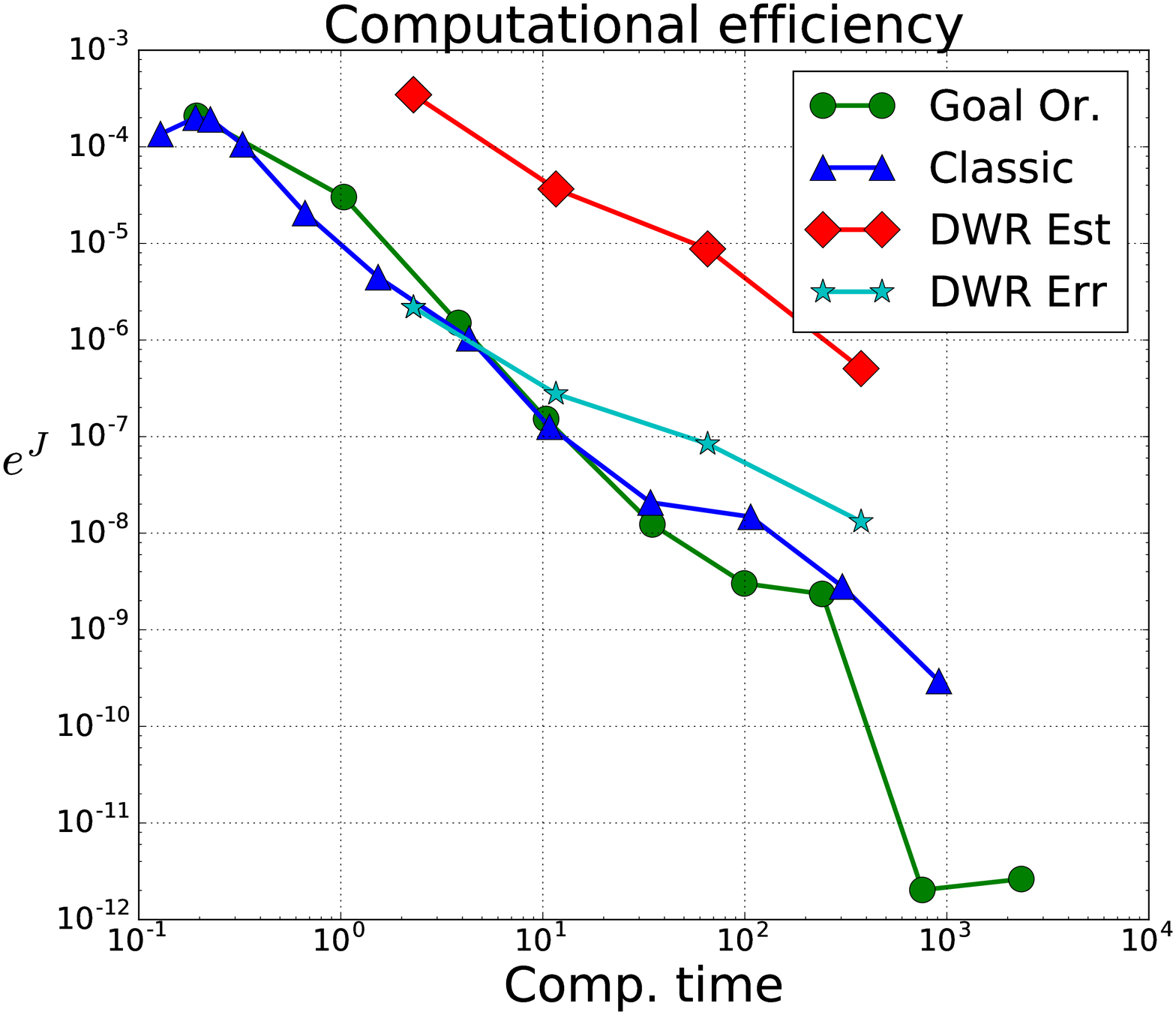}
\includegraphics[scale = 0.19]{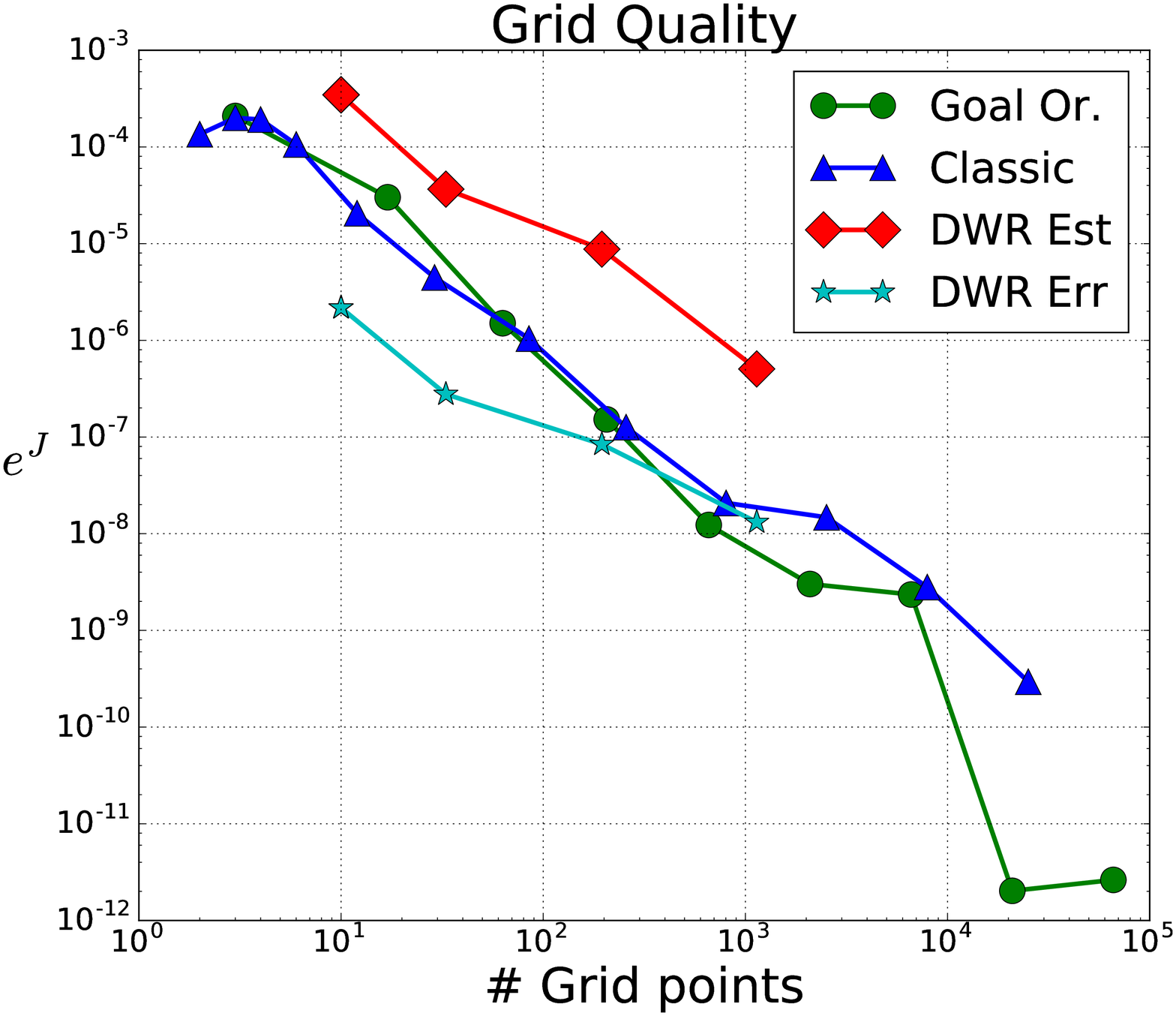}
\caption{Performance comparison of the various methods for problem \eqref{EQ HEAT PROP} and QoI \eqref{EQ HEAT J}. We changed the sign in $\bm{v}$ and use $t_e = 3$.}
\label{FIG HEAT BWD}
\end{figure}

The goal oriented method performs well in this example, but not better than the classic one. While not properly resolving the source term does allow larger timesteps, it does not seem to yield an advantage in terms of computational efficiency or grid quality. The DWR method performs better than in the previous examples, but is still slower than the local error based methods.
%
\subsection{Coupled Heat equations}
%
As a third test problem we consider the coupling of two heat equations with different thermal conductivities and diffusivities. As QoI we choose the average heat transfer over their interface $\Gamma$. The model equations for this problem are
\begin{align}
\alpha_m \frac{\partial u_m(t, \bm{x})}{\partial t} - \nabla \cdot (\lambda_m \nabla u_m(t, \bm{x})) = 0, &\quad (t,\bm{x}) \in [t_0, t_e] \times \Omega_m,\quad m = 1,2,\nonumber\\
u(t, x) = 0, &\quad (t,\bm{x}) \in [t_0, t_e] \times \Omega_m\setminus\Gamma,\nonumber\\
u_1(t, \bm{x}) = u_2(t, \bm{x}), \quad 
\lambda_2 \frac{\partial u_2(t, \bm{x})}{\partial \bm{n}_2} = - \lambda_1 \frac{\partial u_1(t, \bm{x})}{\partial \bm{n}_1}, & \quad (t, \bm{x}) \in [t_0, t_e] \times \Gamma, \label{EQ FSI TRANS COND}\\
u_m(t_0, \bm{x}) = u_m^0(\bm{x}), & \quad\bm{x} \in \Omega_m, \quad m = 1,2.\nonumber
\end{align}
The QoI
\begin{equation*}
J(u) 
= \int_{t_0}^{t_e}\int_{\Gamma} \frac{1}{t_e - t_0} \lambda_1 \frac{\partial u_1(t, \bm{x})}{\partial \bm{n}_1}\,\, dx\, dt
= - \int_{t_0}^{t_e}\int_{\Gamma} \frac{1}{t_e - t_0} \lambda_2 \frac{\partial u_2(t, \bm{x})}{\partial \bm{n}_2}\,\, dx\, dt,,
\end{equation*}
describes the time-averaged heat transfer over the interface. We consider the spatial domains $\Omega_1 = [0,1]\times[0,1]$, $\Omega_2 = [1,2]\times[0,1]$, $\Gamma = \Omega_1 \cap \Omega_2$ and $[t_0, t_e] = [0, 1]$. For discretization in space we use standard linear finite elements for both domains with identical triangular meshes for $\Delta x = 1/21$. In the discrete case the QoI becomes a summed finite difference, which we calculate based on the solution in $\Omega_1$.

For time-integration we use the SDIRK2 scheme, which is implicit with $(p,\, \hat{p}) = (2,\,1)$. To solve the problem arising from the so called transmission conditions \eqref{EQ FSI TRANS COND} on the interface $\Gamma$, we use the Dirichlet-Neumann iteration for each stage derivate of SDIRK2 \cite{Birken2010}. In the heat equations we choose the parameters $\alpha_1 = 0.6 ,\,\,\lambda_1 = 0.3$ and $\alpha_2 = \lambda_2 = 1$. Based on the results of \cite{monge2017}, this gives us a convergence rate of approximately $ \alpha_1/\alpha_2 = 0.6$ for the Dirichlet-Neumann iteration for $\Delta t \rightarrow 0$. The cancellation criterion for the Dirichlet-Neumann iteration is based on the update between two iterates for which we use a tolerance of $10^{-10}$, such that the arising error does not exceed the local errors.

Implementation of the discretization and methods are thanks to Azahar Monge, more details on the discretization in space are found in \cite{monge2016}. We use the initial timestep $\Delta t_0 = \tau$ for all computations. As our reference we use the solution from the classical adaptive method with $\tau = 10^{-7}$.
\subsubsection{Method comparison and performance tests}
Based on the results from the previous problems, we no longer consider the DWR method. While we specifically considered both grid quality and computational efficiency because of the DWR method, we now look only at grid quality, as these two performance measures are essentially identical here.

As our problem has zero Dirichlet boundary conditions and no source term, the solution will vanish. The question is how much heat transfer over the interface will occur during this process.

We consider the problem for two different sets of initial conditions given by
\begin{align}
u^0_1(\bm{x}) &= \left|200 \sin(\pi/2\,x_1^2)\,\sin(\pi\,x_2^2)\right|, \label{EQ FSI U0 1}\\
u^0_2(\bm{x}) &= 
\begin{cases} 
200, & \bm{x} \in \Omega_2/\partial \Omega,\\
0, & \text{otherwise}.
\end{cases}\label{EQ FSI U0 2}
\end{align}
\begin{figure}[ht!]
\centering
\includegraphics[scale = 0.19]{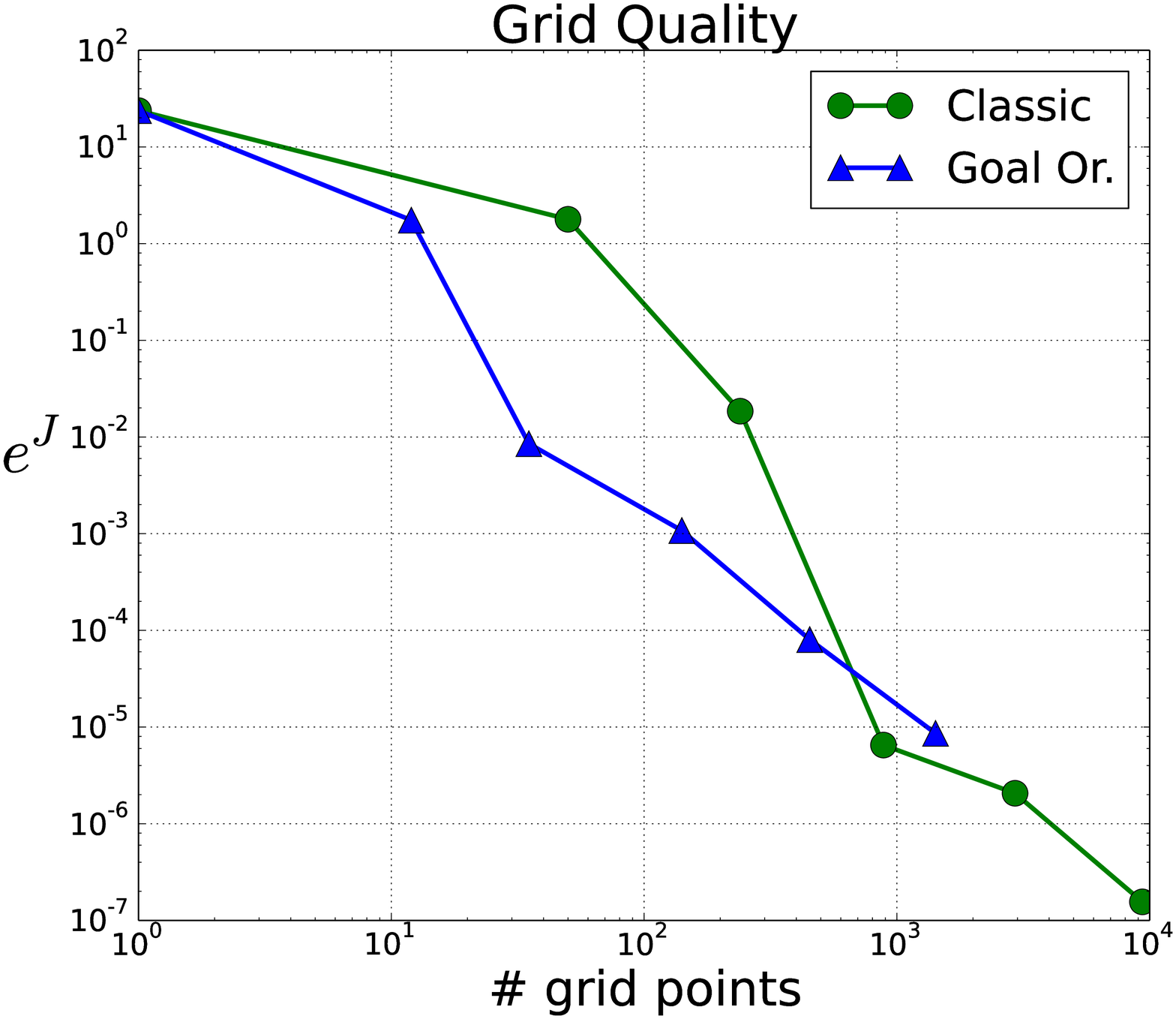}
\includegraphics[scale = 0.19]{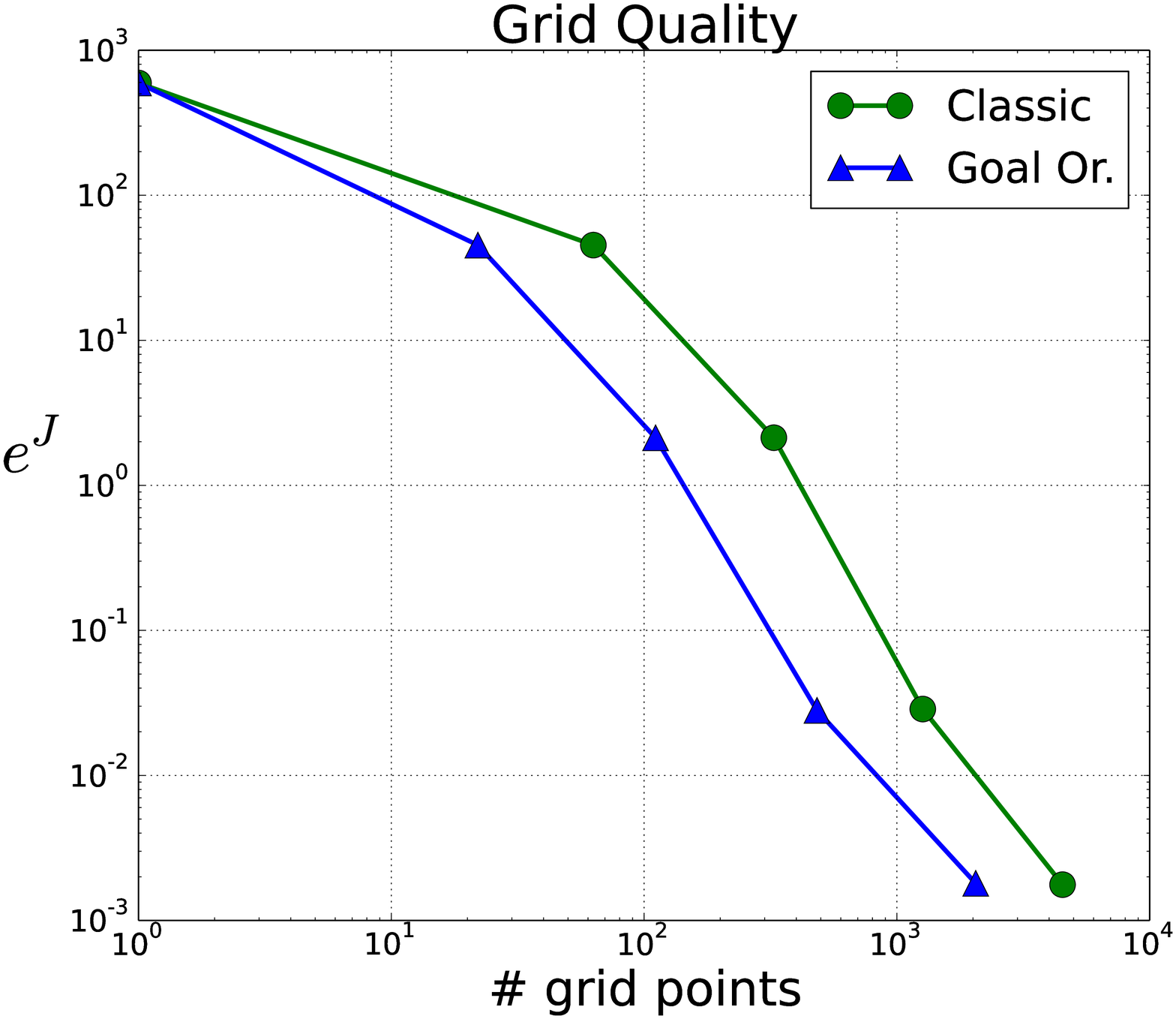}
\caption{Performance comparison of the local error based adaptive methods for the coupled heat problem with initial conditions \eqref{EQ FSI U0 1} on the left and \eqref{EQ FSI U0 2} on the right.}
\label{FIG FSI GRID}
\end{figure}
With \eqref{EQ FSI U0 1}, the initial conditions at the interface $x_1 = 1$ are symmetric, but the steepest heat gradient is inside $\Omega_2$. The choice of timesteps of the classical method is governed by the internal dynamics of $\Omega_2$, whereas the goal-oriented method will choose larger timesteps, especially in the beginning. Not correctly resolving the internal dynamics of $\Omega_2$ does, however, not have a big impact on the values at the interface, due to diffusion. The results in Figure \ref{FIG FSI GRID} (left) show that the grid quality of both methods are on the same parameterized curve for sufficiently small tolerances. For the same tolerance, the classical method gives a smaller error, since it resolves the internal dynamics of $\Omega_2$.

For the initial condition \eqref{EQ FSI U0 2}, the heat transfer over the interface is a good measure of the speed of the diffusion process. While the heat gradient is likely to be steeper at the non-interface boundaries of $\Omega_2$, these areas have little to no impact on our QoI. Hence the classical method will choose smaller timesteps than necessary for the QoI. This is confirmed by the results in Figure \ref{FIG FSI GRID} (right), which show that the goal oriented method performs better.

The timesteps over time in Figure \ref{FIG FSI TIMESTEPS} show that for the initial conditions \eqref{EQ FSI U0 1}, the chosen timesteps have a similar shape, but are shifted. This explains that the performance for both methods lie on the same curve for $\tau \rightarrow 0$. However, for the initial condition \eqref{EQ FSI U0 2} the timesteps have a different shape, one which gives better performance.

\begin{figure}[ht!]
\centering
\includegraphics[scale = 0.19]{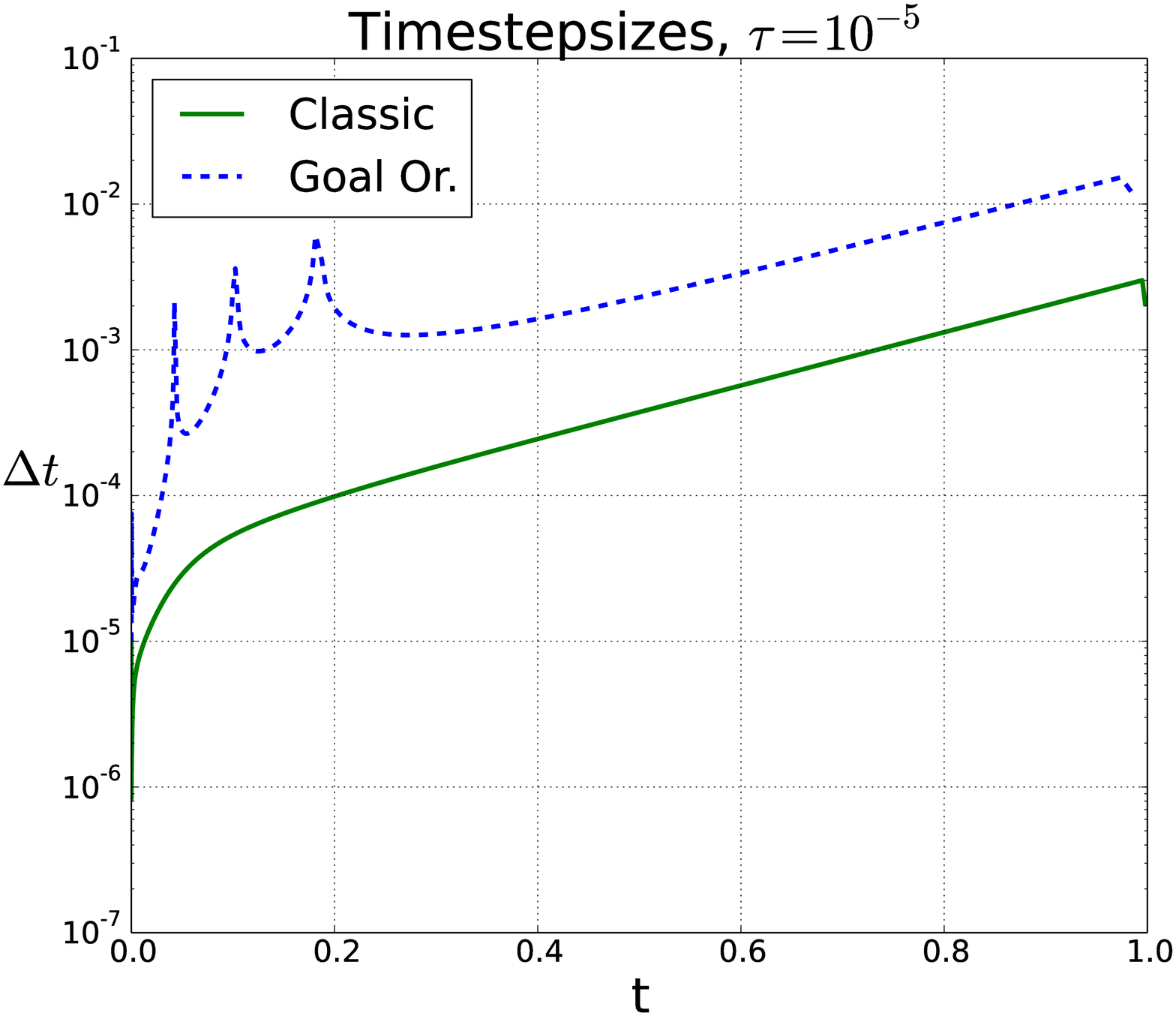}
\includegraphics[scale = 0.19]{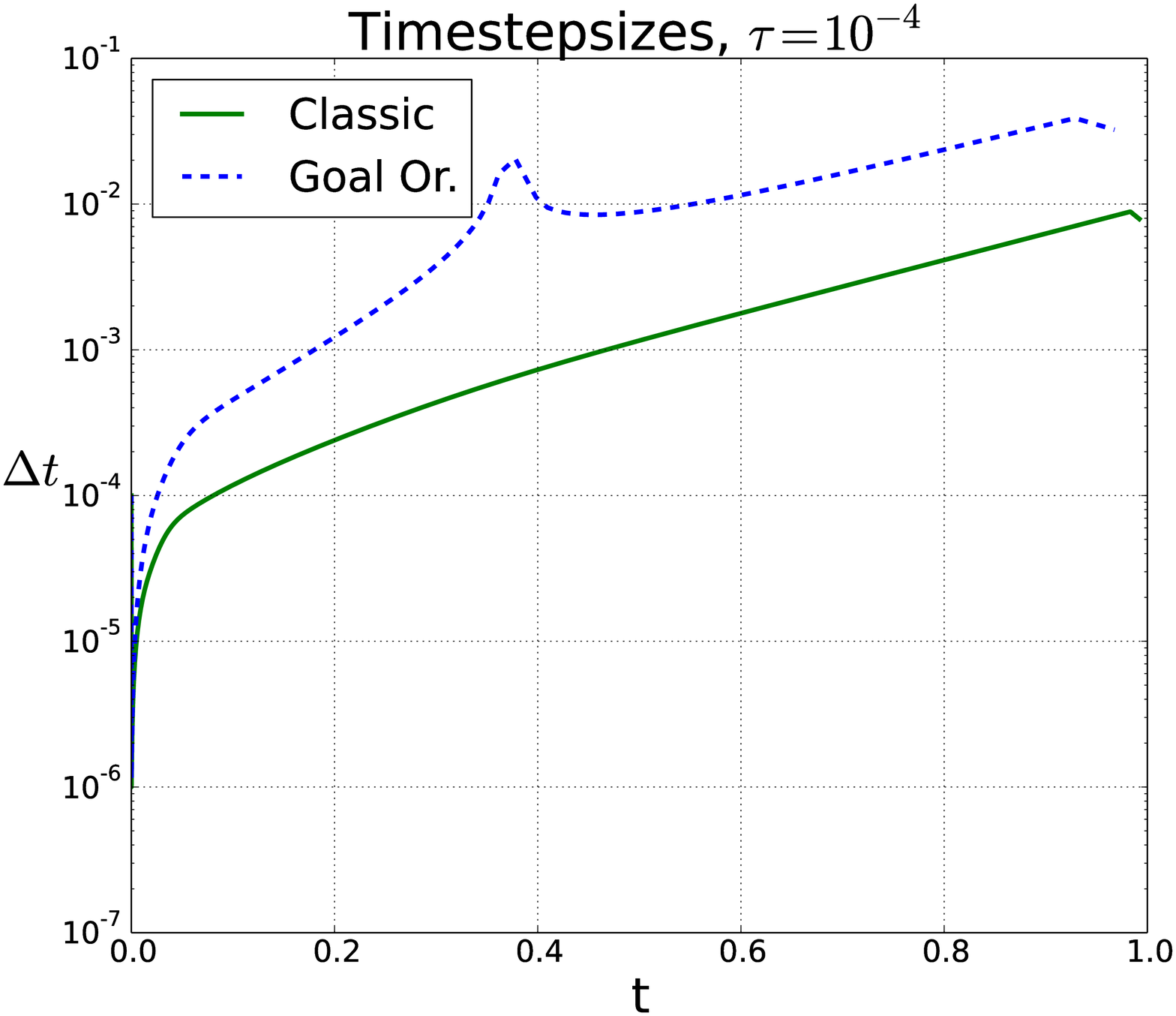}
\caption{Timestep series from solving the coupled heat problem with initial conditions \eqref{EQ FSI U0 1} for $\tau = 10^{-5}$ on the left and \eqref{EQ FSI U0 2} for $\tau = 10^{-4}$ on the right.}
\label{FIG FSI TIMESTEPS}
\end{figure}
%
%
\section{Conclusions}
%
%
We derived a simple and easy to implement goal oriented local error estimator. For the resulting goal oriented adaptive method we prove convergence in the QoI. The constructive nature of our proof gives us necessary requirements for convergence and on $\Delta t_0$. Specifically, we require the error estimate to be non-zero at all times. While this is a natural assumption on controllability, one has to keep in mind that the error estimate is not based on a norm and can have a non-trivial nullspace.

A broad range of initial timesteps are allowed, as long as they are of the right order with respect to the tolerance. This means our results hold for any reasonable scheme used to compute initial timesteps.

Furthermore we show convergence rates and sufficient requirements on the involved schemes to get high convergence rates in the QoI. This involves the need for high order solutions in the quadrature evaluation points, for which we describe how to get the right coefficients for RK schemes. The structure of our proof allows to immediately conclude the same result for closely related controllers. 

We further derived guidelines to predict performance of the goal oriented method in relation to classical adaptive methods. These are based on analyzing global error propagation with respect to the nullspace of the error estimator. The goal oriented adaptive method does not regard errors in the nullspace of the error estimator when choosing timesteps. If processes in the nullspace are not sufficiently resolved by the chosen timesteps and the resulting error affects the QoI, due to global error propagation, performance of the goal oriented method will suffer. The goal oriented method will perform well, if all relevant processes are sufficiently resolved. To use these guidelines one requires sufficient knowledge of the global error dynamics of a problem.

In numerical experiments designed to test these guidelines, we confirm the results on convergence rates and that the guidelines hold true for our test-cases. We test a linear system with two variables with varying stiffness for various QoIs. As more complex test cases we have a 2D convection diffusion equation with source term that is outside the QoI. Further we test two coupled heat equations with varying coefficients and have heat transfer over the interface as the QoI.

The tests show that it is easy to correctly predict bad performance of the goal oriented method. It is, however, hard to predict if the goal oriented method will perform better than a classical norm-based adaptive method.

The results further show that the local error based adaptive methods perform better than the DWR method. The goal oriented method is shown to perform well in many cases, it is, however, not recommended to use it as a black-box solver for general goal oriented problems.
%
\section*{Acknowledgements}
The authors want to thank Patrick Farrell for helping with FEniCS and dolfin-adjoint, Claus F\"uhrer and Gustaf S\"oderlind for many interesting discussions and feedback, and Azahar Monge for the implementation of the final test problem.
%
\bibliographystyle{siam}
\bibliography{MeisrimelBirkenrefs}
\end{document}